\documentclass[11pt]{article}

\usepackage{graphicx}
\usepackage{epsfig}
\usepackage{amsmath, amsthm, amssymb, mathtools}
\usepackage[mathscr]{euscript}
\usepackage{ifthen}
\usepackage[margin=1.5in]{geometry}
\usepackage[english]{babel}
\usepackage[T1]{fontenc}
\usepackage{aliascnt}
\usepackage{hyperref}
\usepackage[all]{xy}
\usepackage{url}
\usepackage{tikz-cd}
\usepackage[final]{showlabels}
\usepackage[capitalize]{cleveref}

\newcounter{temp}

\def\ds{\displaystyle}

\newcommand\Gal{\mathrm{Gal}}
\def\Hom{\mathrm{Hom}}

\def\Bij{\mathrm{Bij}}
\def\Map{\mathrm{Map}}

\def\Gal{\mathrm{Gal}}
\def\det{\mathrm{det}}

\def\trace{\mathrm{Tr}}
\newcommand\id{\operatorname{id}}

\def\Z{\mathbb{Z}}
\newcommand{\fundgroup}[1]{\pi_{#1}}

\newcommand{\F}[1]{\ensuremath{\mathbb{F}_{#1}}}
\newcommand{\extpower}{{\textstyle\bigwedge}}
\renewcommand{\subset}{\subseteq}
\renewcommand{\supset}{\supseteq}
\renewcommand{\phi}{\varphi}

\def\simto{\ds\mathop{\longrightarrow}^\sim\,}

\newcommand{\charpoly}[1]{\mathrm{p}_{#1}}
\newcommand{\Ferrand}[2][]{\ifthenelse{\equal{#1}{}}{\Phi_{#2}}{\Phi_{#2/#1}}}
\newcommand{\norm}[2][]{\ifthenelse{\equal{#1}{}}{\operatorname{Nm}_{#2}}{\operatorname{Nm}_{#2/#1}}}
\newcommand{\spec}{\operatorname{Spec}}
\def\into{\hookrightarrow}
\def\onto{\twoheadrightarrow}

\newcommand{\set}[1]{\{1,\dots,#1\}}
\newcommand{\card}[1]{|#1|}
\newcommand{\basis}{\operatorname{e}}
\newcommand{\fixpower}[4][]{(#2^{\otimes_{#1} #3})^{#4}}
\newcommand{\conjugate}[2]{#1^{(#2)}}
\newcommand{\permgroup}[1]{\mathrm{S}_{#1}}
\newcommand{\altgroup}[1]{\mathrm{A}_{#1}}

\newcommand{\dihedralgroup}[1]{\mathrm{D}_{#1}}
\newcommand{\cyclicgroup}[1]{\mathrm{C}_{#1}}
\newcommand{\kleinfour}{\mathrm{V}_4}
\newcommand{\normalizer}[2][]{\mathrm{N}_{#1}(#2)}

\newcommand{\discalg}[1]{\Delta_{#1}}

\newcommand\midotimes{\mathop{\textstyle\bigotimes}}
\newcommand\gclose[4][]{#2^{\otimes_{#1} #3}/#4}

\newtheorem{theorem}{Theorem}[section] % numbered by section
\crefname{theorem}{Theorem}{Theorems}

\newtheorem{lemma}[theorem]{Lemma}
\crefname{lemma}{Lemma}{Lemmas}

\newtheorem{proposition}[theorem]{Proposition}
\crefname{proposition}{Proposition}{Propositions}

\newtheorem{corollary}[theorem]{Corollary}
\crefname{corollary}{Corollary}{Corollaries}

\theoremstyle{definition} % styled differently: not italicized

\newtheorem{definition}[theorem]{Definition}
\crefname{definition}{Definition}{Definitions}

\newtheorem{remark}[theorem]{Remark}
\crefname{remark}{Remark}{Remarks}

\newtheorem{example}[theorem]{Example}
\crefname{example}{Example}{Examples}

\author{Owen Biesel}
\begin{document}
\title{Galois closure data for extensions of rings}
\maketitle
\abstract{To generalize the notion of Galois closure for separable field extensions, we devise a notion of $G$-closure for algebras of commutative rings $R\to A$, where $A$ is locally free of rank $n$ as an $R$-module and $G$ is a subgroup of $\permgroup{n}$. 
A \emph{$G$-closure datum} for $A$ over $R$ is an $R$-algebra homomorphism $\phi:\fixpower{A}{n}{G}\to R$ satisfying certain properties, and we associate to a closure datum $\phi$ a \emph{closure algebra} $A^{\otimes n}\otimes_{\fixpower{A}{n}{G}} R$.
This construction reproduces the normal closure of a finite separable field extension if $G$ is the corresponding Galois group.
We describe G-closure data and algebras of finite \'etale algebras over a general connected ring $R$ in terms of the corresponding finite sets with continuous actions by the \'etale fundamental group of $R$.
We show that if $2$ is invertible, then $\altgroup{n}$-closure data for free extensions correspond to square roots of the discriminant, and that $\dihedralgroup{4}$-closure data for quartic monogenic extensions correspond to roots of the cubic resolvent.
This is an updated and revised version of the author's Ph.D.\ thesis.}
\tableofcontents

\section{Introduction}

%Motivation

In Manjul Bhargava's groundbreaking series \emph{Higher composition laws}, he introduced an operation on rank-$n$ $\Z$-algebras called \emph{the $\permgroup{n}$-closure} in order to parameterize cubic, quartic, and quintic rings (see \cite[2.1]{Bha04III}).
Later, Bhargava and Matthew Satriano extended this operation (modulo allowing torsion) in \cite{Bha14} to rank-$n$ algebras over arbitrary base rings; their new $S_n$-closure operation reduces to the Galois closure for finite separable field extensions with associated Galois group $S_n$, and also commutes with base change.  
The paper closes by asking whether similar $G$-closure operations exist for other permutation groups $G\subset S_n$.  
We answer Yes, provided that one is willing to parameterize such $G$-closures with what we call \emph{$G$-closure data}. 

%Basic ideas

If $L$ is a degree-$n$ separable extension of a field $K$ with separable closure $\bar K$, we can define the Galois closure $N$ of $L$ to be the minimal subfield of $\bar K$ containing the images of all field homomorphisms $L\to \bar K$.
Then the Galois group $G = \Gal(N/K)$ is a permutation group: it comes with an action on the $n$-element set of homomorphisms $L\to N$ over $K$.
Choosing an ordering of these $n$ homomorphisms, we identify $G$ with a subgroup of $\permgroup{n}$ and can compile the $n$ homomorphisms into a single $G$-equivariant $K$-algebra homomorphism
\[
L^{\otimes n}:= L\otimes_K L \otimes_K \dots \otimes_K L\to N,
\]
where $G$ acts on $N$ by definition but on $L^{\otimes n}$ via the action of $\permgroup{n}$ on the tensor factors.
In particular, we obtain a homomorphism between the $G$-invariants of the two $K$-algebras,
\[
\phi: \fixpower{L}{n}{G}\to N^G = K.
\]
We will later see that this homomorphism is a \emph{$G$-closure datum} for $L$ over $K$, because of where it sends certain elements of $(L^{\otimes n})^G$.
For each element $\ell\in L$, denote the $k$th elementary symmetric polynomial in the elements
\begin{align*}
 \conjugate{\ell}{1} &= \ell\otimes 1\otimes 1\otimes\dots\otimes 1\\
 \conjugate{\ell}{2} &= 1\otimes\ell\otimes 1\otimes\dots\otimes 1\\
 \dots &\\
 \conjugate{\ell}{n} &= 1\otimes1\otimes\dots\otimes1\otimes \ell
\end{align*}
by $e_k(\ell)$.  
Then $\phi(e_k(\ell))$ is the $k$th elementary symmetric polynomial in the $n$ conjugates of $\ell$ in $N$.
In particular, $\phi$ sends $e_1(\ell)$ to the sum of $\ell$'s $n$ conjugates, that is, the trace of $\ell$.
 And $\phi(e_n(\ell))$ is the product of $\ell$'s $n$ conjugates, the norm of $\ell$.

More generally, if $\ell$ is an element of $L$, we can regard multiplication by $\ell$ as a $K$-linear map $L\to L$.  
This linear map corresponds to an $n\times n$ matrix $M_\ell$ with entries in $K$, for each choice of $K$-basis for $L$.  
The characteristic polynomial of $M$ is independent of this choice of basis, so its coefficients are elements of $K$ that depend only on $\ell$.  
Write this characteristic polynomial as
\[\charpoly{\ell}(\lambda) := \det(\lambda I - M_\ell) = \lambda^n - s_1(\ell)\lambda^{n-1} + s_2(\ell)\lambda^{n-2} - \ldots + (-1)^n s_n(\ell).\]
Then the homomorphism $\phi: (L^{\otimes n})^G\to K$ sends $e_k(\ell)$ to $s_k(\ell)$ for each $k\in\set{n}$; this is the defining feature of a closure datum.

Namely, the concept of characteristic polynomial extends to the following setting: let $R$ be a ring and $A$ a rank-$n$ $R$-algebra, i.e.\ an algebra that is locally free of rank $n$ as an $R$-module.
(Note: in this paper all rings and algebras in this paper are commutative by assumption; thus an $R$-algebra is just a ring $A$ with a ring homomorphism $R\to A$.)
Then for each $a\in A$, the multiplication-by-$a$ homomorphism $A\to A$ locally corresponds to action of an $n\times n$-matrix on $R^n$, and the characteristic polynomials of these matrices glue to a well-defined polynomial with coefficients in $R$, which we again write as
\[\charpoly{a}(\lambda) = \lambda^n - s_1(a)\lambda^{n-1} + s_2(a)\lambda^{n-2} - \ldots + (-1)^n s_n(a).\]
Then a \emph{$G$-closure datum for $A$ over $R$} is an $R$-algebra homomorphism $\phi:\fixpower{A}{n}{G}\to R$ that sends $e_k(a)$ to $s_k(a)$ for each $a\in A$ and each $k\in\set{n}$.

A given ring $R$ and algebra $A$ may have $G$-closure data for only some groups $G$; we will see in \cref{main-etale} that if $K\into L$ is a finite separable field extension, the only closure data are the above $\phi$ for various orderings of $\Hom_K(L,N)$, along with their restrictions to the algebras of invariants under larger subgroups of $\permgroup{n}$.

Now given such a $\phi:\fixpower{L}{n}{G}\to K$ coming from a finite separable field extension $K\into L$ with Galois closure $N$, consider the following commutative square:
\[\begin{tikzcd}
 \fixpower{L}{n}{G} \arrow{r}{\phi}\arrow{d} & K \arrow{d}\\
 L^{\otimes n} \arrow{r} & N
\end{tikzcd}\]
In fact, the square is a tensor product diagram, that is, $L^{\otimes n}\otimes_{\fixpower{L}{n}{G}} K\cong N$.
In general, given a $G$-closure datum $\phi$ for $A$ over $R$, we will associate to it the \emph{closure algebra} $A^{\otimes n}\otimes_{\fixpower{A}{n}{G}} R$, thus generalizing the normal closure in the case of fields.

%Organization

The organization of this paper is as follows:
In \cref{section_definitions}, we phrase the definition of closure datum in terms of the \emph{Ferrand homomorphism} associated to a rank-$n$ algebra, in order to make several results easier to prove.
In \cref{section_relationships}, we consider the various ways of producing some closure data from others, by varying the group, base ring, or algebra.
We also introduce a notion of two closure data being isomorphic, and discuss the question of when a minimal closure datum is unique up to isomorphism.
\Cref{section_etale} demonstrates that this theory of closure data reduces to ordinary Galois theory for finite \'etale algebras.
In \cref{section_products} we show that given a $G$-closure datum on an $R$-algebra $A$, and an $H$-closure datum on an $R$-algebra $B$, we get a $G\times H$-closure datum on $A\times B$, and that the resulting closure algebra is the tensor product of those for $A$ and $B$.

In \cref{an-closures} we use the results of \cite{Bie15} to classify $\altgroup{n}$-closure data.
Namely, for each rank-$n$ $R$-algebra $A$ there is a rank-$2$ $R$-algebra $\discalg{A/R}$ such that $\altgroup{n}$-closure data for $A$ over $R$ are in one-to-one correspondence with $R$-algebra homomorphisms $\discalg{A/R}\to R$.
The main theorem of that section is \cref{sqrt-disc}, which implies that if $R$ is a ring in which $2$ is a unit, and if $A$ is a rank-$n$ $R$-algebra that is free as an $R$-module, then $A$ has an $\altgroup{n}$-closure datum if and only if the discriminant of $A$ is a square in $R$.

Finally, in \cref{section_monogenic} we explore the case of \emph{monogenic} algebras, i.e.\ those rank-$n$ $R$-algebras of the form $A\cong R[x]/(f(x))$.
Cataloguing the $G$-closure data for such an $A$ over $R$ is analogous to identifying the Galois group of $f$.
We show first that if $G$ is an intransitive permutation group, then every $G$-closure datum for $A$ over $R$ yields a nontrivial factorization of $f$; thus if $f$ is irreducible then all closure data are with respect to transitive subgroups of $\permgroup{n}$.
Finally, we show that under mild hypotheses on $R$ and $G$, the $G$-closure data for $A$ over $R$ correspond to homomorphisms $\fixpower{R[x]}{n}{G}\to R$ sending each $e_k(x)$ to $s_k(a)$, where $a\in A$ corresponds to the element $x\in R[x]/(f(x))$.
We then use this correspondence to show that $\dihedralgroup{4}$-closure data for a quartic polynomial $f$ correspond bijectively to roots of $f$'s \emph{cubic resolvent}, with no assumptions on the base ring $R$.

\begin{remark}
This paper is a condensed version of the author's Ph.D.\ thesis \cite{Bie13} under Manjul Bhargava, and as such, the thanks expressed there to him and everyone else who provided invaluable support to that project still apply.
However, the sections on universally norm-preserving homomorphisms and product algebras in this paper are new, and nearly every definition, statement, and proof has been revised from its original version.
The author would therefore like to thank Maarten Derickx, Bas Edixhoven, Alberto Gioia, Lenny Taelman, and Hendrik Lenstra for their many helpful insights during this revision process. 
\end{remark}

\section{The Ferrand homomorphism and closure data}
\label{section_definitions}

%%%%%% Definition of locally free

We begin by reviewing the context in which our variant on Galois closures makes sense:

\begin{definition}
 Let $R$ be a ring, $M$ an $R$-module, and $n$ a natural number.  
 We say that $M$ is \emph{locally free of rank $n$} if the unit ideal of $R$ is generated by the set of all $r\in R$ such that the localization $M_r$ is free of rank $n$ as an $R_r$-module.
 Such a module is automatically projective and finitely generated.
 An \emph{$R$-algebra of rank $n$} is an $R$-algebra that is locally free of rank $n$ as an $R$-module.
\end{definition}

%%%%%% Definition of Ferrand

Recall from \cite[Def.\ 2.5]{Bie15} that for each pair $(R,A)$ with $R$ a ring and $A$ an $R$-algebra of rank $n$, there is a canonical $R$-algebra homomorphism $\fixpower{A}{n}{\permgroup{n}}\to R$, which is denoted $\Ferrand[R]{A}$ and called the \emph{Ferrand homomorphism} for $A$ over $R$.
Together, the Ferrand homomorphisms for various $R$ and $A$ have these properties:
\begin{enumerate}
 \item For each ring $R$ and $R$-algebra $A$ of rank $n$, and for each $a\in A$, we have
 \[\Ferrand[R]{A}(a\otimes\dots\otimes a) = \norm[R]{A}(a),\]
 the norm of $a$, i.e.\ the element of $R$ such that for every $a_1\wedge\dots\wedge a_n\in\extpower^n A$, we have
 \[(aa_1)\wedge(aa_2)\wedge\dots\wedge(aa_n) = \norm[R]{A}(a) (a_1\wedge\dots\wedge a_n).\]
 \item The Ferrand homomorphisms \emph{commute with base change}: if $R$ is a ring and $A$ is an $R$-algebra of rank $n$, and if furthermore $R'$ is any $R$-algebra and we denote the rank-$n$ $R'$-algebra $R'\otimes_R A$ by $A'$, then the following square of $R'$-algebra homomorphisms commutes:
 \[\begin{tikzcd}
  R'\otimes_R\fixpower{A}{n}{\permgroup{n}} \arrow{r}{\sim}\arrow{d}[swap]{{\id_{R'}\otimes\, \Ferrand[R]{A}}} 
  & \fixpower[R']{A'}{n}{\permgroup{n}}\arrow{d}{{\Ferrand[R']{A'}}}
  \\
  R'\otimes_R R \arrow{r}{\sim} & R'
   \end{tikzcd}\]
\end{enumerate}

%%%%%% Discussion of Ferrand homomorphism

As a consequence of this abstract characterization of the Ferrand homomorphisms, we find that for each element $a$ of a rank-$n$ $R$-algebra $A$, the Ferrand homomorphism $\Ferrand[R]{A}$ applied coefficientwise to $(x-a)\otimes(x-a)\otimes\dots\otimes(x-a)$ gives the characteristic polynomial $\charpoly{a}(x) := \norm[{R[x]}]{A[x]}(x-a)$ of $a$.
 In other words,
\begin{lemma}\label{ek-to-sk}
Denote the coefficient of $(-1)^k x^{n-k}$ in $(x-a)\otimes\dots\otimes(x-a)$ by $e_k(a)$, the $k$-th elementary symmetric polynomial evaluated in the $n$ elements $\conjugate{a}{1} = a\otimes 1\otimes\dots\otimes 1$ up to $\conjugate{a}{n} = 1\otimes\dots\otimes 1\otimes a$ in $A^{\otimes n}$.
 And denote the coefficient of $(-1)^k x^{n-k}$ in $\norm[{R[x]}]{A[x]}(x-a) = \charpoly{a}(x)$ by $s_k(a)\in R$.
 Then
 \[\Ferrand[R]{A}\bigl(e_k(a)\bigr) = s_k(a).\]
\end{lemma}
Properties 1 and 2 above identify the Ferrand homomorphisms as those $R$-algebra homorphisms $\fixpower{A}{n}{\permgroup{n}}\cong\Gamma_R^n(A)\to R$ corresponding to the multiplicative degree-$n$ polynomial law $\norm[R]{A}\colon A\to R$; this is Ferrand's original characterization of the Ferrand homomorphisms in \cite{Fer98}.
But by the following lemma, the description in \cref{ek-to-sk} also uniquely characterizes the Ferrand homomorphism $\Ferrand[R]{A}$ for each ring $R$ and $R$-algebra $A$:

\begin{lemma}\label{fixpower-generators}
 Let $R$ be a ring and $A$ a projective finitely-generated $R$-algebra.
 Then $\{e_k(a): a\in A\text{ and }k\in\{1,\dots,n\}\}$ generates $\fixpower{A}{n}{\permgroup{n}}$ as an $R$-algebra.
\end{lemma}
\begin{proof}
 The case of $A = R[x_1,\dots,x_m]$ is covered by \cite[Theorem 1]{Vac05}.
 For a general projective $R$-algebra $A$ generated by a finite set $\{a_1,\dots,a_m\}$, form a surjection $R[x_1,\dots,x_m]\onto A$ sending each $x_i$ to $a_i$.
 Since $A$ is projective as an $R$-module, this surjection has a right inverse in the category of $R$-modules.
 Applying the $\fixpower{(\cdot)}{n}{\permgroup{n}}$ endofunctor of $R$-modules, then, we find that the $R$-algebra homomorphism $\fixpower{R[x_1,\dots,x_m]}{n}{\permgroup{n}}\to \fixpower{A}{n}{\permgroup{n}}$ is also a surjection.
 Since $\fixpower{R[x_1,\dots,x_m]}{n}{\permgroup{n}}$ is generated by elements of the form $e_k(f)$ with $f\in R[x_1,\dots,x_m]$, its quotient $R$-algebra $\fixpower{A}{n}{\permgroup{n}}$ is generated by the collection of their images, namely 
 $\{e_k(f(a_1,\dots,a_m)): k\in\set{n}\text{ and }f\in R[x_1,\dots,x_m]\} = \{e_k(a): k\in\set{n}\text{ and }a\in A\}$.
\end{proof}

\begin{corollary}\label{ferrand-uniqueness}
 Let $R$ be a ring and $A$ an $R$-algebra of rank $n$.
 The Ferrand homomorphism $\Ferrand[R]{A}\colon \fixpower{A}{n}{\permgroup{n}}\to R$ is the unique $R$-algebra homomorphism that sends $e_k(a)$ to $s_k(a)$ for all $a\in A$ and $k\in\set{n}$.
\end{corollary}

%%%%%% Definition of closure data

We now state the main definition of this paper:

\begin{definition}
 Let $R$ be a ring and $A$ a rank-$n$ $R$-algebra.  
 A \emph{closure datum for $A$ over $R$} is a pair $(G,\phi)$, where $G$ is a subgroup of $\permgroup{n}$ and $\phi$ is an $R$-algebra homomorphism $\fixpower{A}{n}{G}\to R$ that extends the Ferrand homomorphism $\Ferrand[R]{A}\colon \fixpower{A}{n}{\permgroup{n}}\to R$.
 
If $(G,\phi)$ is a closure datum for $A$ over $R$, then we also say that $\phi$ is a \emph{$G$-closure datum for $A$ over $R$}.
Furthermore, given a $G$-closure datum $\phi$, we denote the $R$-algebra given by the tensor product
\[ A^{\otimes n} \midotimes_{\mathclap{\fixpower{A}{n}{G}}} R\ \cong\,A^{\otimes n}/\bigl(\alpha - \phi(\alpha): \alpha\in \fixpower{A}{n}{G}\bigr)\]
by $\gclose{A}{n}{\phi}$ and call it the \emph{$G$-closure} (or \emph{closure algebra}) of $A$ over $R$ associated with $\phi$. 
 \end{definition}

%%%%%% S_n-closure example

\begin{example}
There's exactly one $\permgroup{n}$-closure for each rank-$n$ $R$-algebra $A$: the one associated with $\Ferrand[R]{A}$ itself.
Its associated $R$-algebra is 
\begin{align*} 
\gclose{A}{n}{\Ferrand[R]{A}}&\cong A^{\otimes n}/\bigl(\alpha - \Ferrand[R]{A}(\alpha): \alpha\in \fixpower{A}{n}{G}\bigr)\\
&= A^{\otimes n}/\bigl(e_k(a) - s_k(a): a\in A\text{ and }k\in\{1,\dots,n\}\bigr),
\end{align*}
where we have used \cref{ferrand-uniqueness} to simplify the ideal presentation.
This last quotient is exactly the $\permgroup{n}$-closure of $A$ over $R$ defined by Bhargava and Satriano in \cite{Bha14}.
\end{example}

%%%%%%

\section{Relationships between closure data}
\label{section_relationships}

In this section we demonstrate some ways of obtaining one closure datum from another, either by replacing the group with a larger subgroup of $\permgroup{n}$, changing the base from $R$ to an arbitrary $R$-algebra $R'$, or pulling back closure data for $B$ to closure data for $A$ along certain $R$-algebra homomorphisms $A\to B$.

\subsection{Inducing closure data}

Varying the group gives us the most elementary means of producing new closure data:

%%%%%% Inducing closure data

\begin{proposition}
 Let $R$ be a ring and $A$ an $R$-algebra of rank $n$.
 If $(G,\phi)$ is a closure datum for $A$ over $R$, and $H$ is a subgroup of $\permgroup{n}$ containing $G$, then $(H, \phi|_{\fixpower{A}{n}{H}})$ is also a closure datum for $A$ over $R$.
\end{proposition}
\begin{proof}
 The only criterion to check is that $\phi|_{\fixpower{A}{n}{H}}$ restricts to the Ferrand homomorphism on $\fixpower{A}{n}{\permgroup{n}}$, but this is true because $\phi$ does.
\end{proof}

We say that the $G$-closure datum $\phi$ \emph{induces} the $H$-closure datum $\phi|_{\fixpower{A}{n}{H}}$ for each $H$ containing $G$.
In particular, the property of a given $R$-algebra having a $G$-closure is upward-closed with respect to $G$.
For this reason, ``having a $G$-closure'' can be thought of as roughly corresponding to ``having Galois group contained in $G$''---the analogue of \emph{the} Galois group, then, is a minimal group $G$ for which a $G$-closure datum exists. 
The following definition enriches this idea with the closure data:

\begin{definition}
A closure datum $(G,\phi)$ for an $R$-algebra $A$ of rank $n$ is called \emph{minimal} if it is not induced by any other closure datum, i.e., if the homomorphism $\phi\colon\fixpower{A}{n}{G}\to R$ cannot be extended to a homomorphism $\fixpower{A}{n}{H}\to R$ for any smaller subgroup $H\subsetneq G$.
\end{definition}

%%%%%% Isomorphisms of closure data

For a given ring and algebra, is there always a unique minimal closure datum?
The answer is typically ``No'' for trivial reasons: if $(G,\phi)$ is a closure datum for an $R$-algebra $A$ of rank $n$, and $\sigma\in\permgroup{n}$ is any permutation, then $(\sigma G \sigma^{-1}, \phi\circ(\id_A)^{\otimes \sigma^{-1}})$ is another closure datum.
(Here $(\id_A)^{\otimes\sigma^{-1}}$ is the automorphism of $A^{\otimes n}$ sending $\conjugate{a}{i}$ to $\conjugate{a}{\sigma^{-1}(i)}$.
It restricts to a map $\fixpower{A}{n}{\sigma G \sigma^{-1}}\to\fixpower{A}{n}{G}$.)
This describes an action of $\permgroup{n}$ on the set of closure data for a given $R$-algebra, and the resulting action groupoid gives us a notion of two closure data being isomorphic:

\begin{definition}
 Let $R$ be a ring and $A$ an $R$-algebra of rank $n$.
 An \emph{isomorphism of closure data} $(G, \phi)\to (H, \psi)$ for $A$ over $R$ is a permutation $\sigma\in\permgroup{n}$ for which $H = \sigma G\sigma^{-1}$ and for which $\psi = \phi\circ(\id_A)^{\otimes\sigma^{-1}}$.
\end{definition}

\begin{remark}\label{minimal-uniqueness}
 The question we should be asking, then, is whether every pair of minimal closure data is isomorphic in this sense.
We will show in \cref{section_etale} that this is the case if the ring is connected and the algebra is \'etale.
There are also more examples; for instance, every free quadratic $R$-algebra has this property if and only if $R$ is a domain, as we will show in \cref{domain-unique}.

The question of how unique minimal closure data need be has also been taken up by Riccardo Ferrario for rank-$4$ algebras in \cite{Fer14}, where he exhibits an example of a quartic algebra with multiple minimal closure data with respect to groups that are not even conjugate.
On the positive side, Maarten Derickx shows in forthcoming work that if $R$ is a characteristic-zero integrally closed domain, then for every finite-rank $R$-algebra $A$ the groups with minimal closure data for $A$ over $R$ are all conjugate.
He also shows that this can fail in positive characteristic.
\end{remark}

Note that with this definition of (iso)morphism of closure data, the operation of taking the closure associated to a closure datum is then a functor:

\begin{proposition}
 Let $R$ be a ring and $A$ an $R$-algebra of rank $n$, and let $\sigma\colon(G,\phi)\to(H,\psi)$ be an isomorphism of closure data for $A$ over $R$.
 Then the automorphism $(\id_A)^{\otimes\sigma}$ of $A^{\otimes n}$ descends to an $R$-algebra isomorphism of the associated closure algebras $\gclose{A}{n}{\phi} \simto \gclose{A}{n}{\psi}$.
\end{proposition}

\begin{proof}
 That $\sigma\colon(G,\phi)\to(H,\psi)$ is an isomorphism of closure data reflects the commutativity of the following diagram:
 \[\begin{tikzcd}
  & A^{\otimes n}\arrow{rr}{(\id_A)^{\otimes\sigma}}[swap]{\sim}%","\sim"']
   & & A^{\otimes n} \\
  \fixpower{A}{n}{G} \arrow[hook]{ru}\arrow{rd}{\phi}\arrow{rr}{(\id_A)^{\otimes\sigma}}[swap]{\sim} & & \fixpower{A}{n}{H}\arrow[hook]{ru}\arrow{rd}{\psi} & \\
   & R \arrow[equals]{rr} & & R
 \end{tikzcd}\]
 Therefore we obtain the desired isomorphism of the associated closure algebras
 \[\begin{tikzcd}
 \gclose{A}{n}{\phi} =\ds A^{\otimes n} \midotimes_{\mathclap{\fixpower{A}{n}{G}}} R \arrow{r}{\sim} & \ds A^{\otimes n} \midotimes_{\mathclap{\fixpower{A}{n}{H}}} R = \gclose{A}{n}{\psi}.\end{tikzcd}\qedhere\]
\end{proof}

\begin{remark}\label{group-actions}
Note that for a fixed ring $R$ with rank-$n$ algebra $A$, and for a fixed subgroup $G\subset\permgroup{n}$, the set of $G$-closure data for $A$ over $R$ carries a natural action by $\{\sigma\in\permgroup{n}:\sigma G\sigma^{-1} = G\}$, the normalizer $\normalizer[\permgroup{n}]{G}$ of $G$.
Those $\sigma$ that belong to $G$ itself act trivially on the set of $G$-closure data, but induce generally non-trivial automorphisms of the associated closure algebras.
To summarize, for a fixed ring $R$ and rank-$n$ algebra $A$:
\begin{itemize}
 \item The set of all closure data for $A$ over $R$ has a natural $\permgroup{n}$-action.
 \item For a fixed group $G\subset \permgroup{n}$, the set of $G$-closure data for $A$ over $R$ carries a natural action by the group $\normalizer[\permgroup{n}]{G}/G$.
 This orbits of this action are precisely the isomorphism classes of $G$-closure data.
 \item For each $G$-closure datum $\phi$, the associated closure algebra $\gclose{A}{n}{\phi}$ carries a natural action by $G$.
\end{itemize}
\end{remark}

\subsection{Universally norm-preserving homomorphisms}

%%%%%% Define universally norm-preserving homomorphism:

We can also produce closure data by pulling it back along \emph{universally norm-preserving} $R$-algebra homomorphisms:

\begin{definition}
 Let $R$ be a ring and $A$ and $B$ be $R$-algebras of rank $n$.
 An $R$-algebra homomorphism $f\colon A\to B$ is called \emph{norm-preserving} if for all $a\in A$, we have $\norm{A}(a) = \norm{B}(f(a))$.
 We say that $f$ is \emph{universally} norm-preserving if for every $R$-algebra $S$, the $S$-algebra homomorphism $\id_{S}\otimes f: S\otimes_R A \to S\otimes_R B$ is norm-preserving.
\end{definition}

%%%%%% Universally norm-preserving homomorphisms commute with Ferrand:

Here are two alternative characterizations of universally norm-preserving homomorphisms:

\begin{lemma}
 Let $R$ be a ring and $f\colon A\to B$ a homomorphism of $R$-algebras of rank $n$.  
 The following are equivalent:
 \begin{enumerate}
 \item The homomorphism $f$ is universally norm preserving.
 \item The following triangle of $R$-algebra homomorphisms commutes:
 \[\begin{tikzcd}[column sep=tiny]
  \fixpower{A}{n}{\permgroup{n}} \arrow{rd}[swap]{{\Ferrand{A}}}\arrow{rr}{\fixpower{f}{n}{\permgroup{n}}} & & \fixpower{B}{n}{\permgroup{n}} \arrow{ld}{{\Ferrand{B}}}\\
   & R & 
 \end{tikzcd}
 \]
 \item The homomorphism $f$ preserves characteristic polynomials, i.e., for all $a\in A$, we have $\charpoly{f(a)}(\lambda) = \charpoly{a}(\lambda)$, or equivalently, $s_k(f(a)) = s_k(a)$ for all $k\in \{1,\dots,n\}$. 
 \end{enumerate}
\end{lemma}
\begin{proof}
 See \cite[Prop. 7.1]{Bie15}.
\end{proof}

\begin{proposition}
 Let $R$ be a ring and $A$ and $B$ be $R$-algebras of rank $n$.
 If $f\colon A\to B$ is a universally norm-preserving homomorphism and $(G,\phi)$ is a closure datum for $B$, then $(G, \phi\circ\fixpower{f}{n}{G})$ is a closure datum for $A$.
 
 Furthermore, the homomorphism of $R$-algebras $f^{\otimes n}\colon A^{\otimes n}\to B^{\otimes n}$ descends to a homomorphism of the $G$-closures $\gclose{A}{n}{(\phi\circ\fixpower{f}{n}{G})}\to \gclose{B}{n}{\phi}$.
\end{proposition}
\begin{proof}
 We just need to check that $\phi\circ\fixpower{f}{n}{G}$, restricted to $\fixpower{A}{n}{\permgroup{n}}$, is the Ferrand homomorphism $\Ferrand{A}$.
 This restriction is $\phi|_{\fixpower{B}{n}{\permgroup{n}}}\circ\fixpower{f}{n}{\permgroup{n}} = \Ferrand{B}\circ\fixpower{f}{n}{\permgroup{n}}$, which equals $\Ferrand{A}$ since $f$ is universally norm-preserving.
 That the described homomorphism of $G$-closures exists is elementary.
\end{proof}

%We close this section with two simple observations about universally norm-preserving homomorphisms:
%
%\begin{lemma}
% Let $A$ and $B$ be two rank-$n$ $R$-algebras and $f\colon A\to B$ a universally norm-preserving homomorphism.
% \begin{enumerate}
%  \item Let $a\in A$. If $f(a)=0$, then $a$ is nilpotent.
%  \item If $A$ is \'etale, then $B$ is \'etale and $f$ is an isomorphism.
% \end{enumerate}
%\end{lemma}
%
%\begin{proof}
%\begin{enumerate}
%\item The characteristic polynomial of $0$ is $\lambda^n$, so if $f(a)=0$ then the characteristic polynomial of $a$ is also $\lambda^n$; hence $a^n=0$.
%
%\item It is sufficient to check in case $A$ and $B$ are both free as $R$-modules.
%Let $(\theta_1,\dots,\theta_n)$ be a free $R$-module $R$-basis for $A$, and $(\phi_1,\dots,\phi_n)$ a basis for $B$,
%and define coefficients $m_{ij}\in R$ by $f(\theta_i) = \sum_j m_{ij}\phi_j$.
%Then because $f$ preserves products and traces, we have
%\[\det\bigl(\trace_B(f(\theta_i)f(\theta_j))\bigr)_{ij} = \det\bigl(\trace_A(\theta_i\theta_j)\bigr)_{ij},\]
%which is a unit in $R$ because $A$ is \'etale. 
%But we also have
%\[\det\bigl(\trace_B(f(\theta_i)f(\theta_j))\bigr)_{ij} = \det\bigl(\trace_B(\phi_i\phi_j)\bigr)_{ij}\cdot\bigl(\det(m_{ij})_{ij}\bigr)^2,\]
%so $\det\bigl(\trace_B(\phi_i\phi_j)\bigr)_{ij}$ is a unit (i.e.\ $B$ is \'etale) and $\det(m_{ij})_{ij}$ is a unit (i.e.\ $f$ is an isomorphism).
%\end{enumerate}
%\end{proof}

\subsection{Base extension of closure data}

%%%%%% Base change preserves closure data

Finally, base change preserves closure data and commutes with forming the closure algebra:

\begin{proposition}\label{g-closure-base-change}
 Let $R$ be a ring and $A$ an $R$-algebra of rank $n$.
 Let $R'$ be any $R$-algebra, and $A'=R'\otimes_R A$ the resulting $R'$-algebra of rank $n$.
 If $(G,\phi)$ is a closure datum for $A$ over $R$, then $(G,\phi')$ is a closure datum for $A'$ over $R'$, where $\phi'$ is the composite homomorphism
 \[\begin{tikzcd}[column sep=large]\phi'\colon \fixpower[R']{A'}{n}{G}\cong R'\otimes_R\fixpower{A}{n}{G}\arrow{r}{\id_{R'}\otimes \phi} & R'\otimes_R R\cong R'.\end{tikzcd}\]
 Furthermore, the canonical isomorphism $R'\otimes_R A^{\otimes n}\cong A'^{\otimes_{R'} n}$ descends to an isomorphism $R'\otimes_R (\gclose{A}{n}{\phi})\cong \gclose[R']{A'}{n}{\phi'}$.
\end{proposition}

\begin{proof}
 First, note that because $A$ is locally free as an $R$-module, the isomorphism $A'^{\otimes_{R'}n}\cong R'\otimes_R (A^{\otimes n})$ does indeed restrict to an isomorphism $\fixpower[R']{A}{n}{G}\cong R'\otimes_R \fixpower{A}{n}{G}$ by \cite[Prop.\ 3.6]{Bie15}, so the definition of $\phi'$ makes sense.
 Then to check that $(G,\phi')$ is a closure datum for $A'$ over $R'$ is to check that $\phi'$ restricts to $\Ferrand[R']{A'}$.
 But this holds because $\Ferrand[R']{A'}\cong \id_{R'}\otimes\,\Ferrand[R]{A}$, and $\phi$ restricts to $\Ferrand[R]{A}$.
 
 The claim that $R'\otimes_R (\gclose{A}{n}{\phi})\cong \gclose[R']{A'}{n}{\phi'}$ is easily checked using the presentation of the $G$-closure as a tensor product:
 \[R'\otimes_R\left(A^{\otimes n}\midotimes_{\mathclap{\fixpower{A}{n}{G}}} R\right)\cong (R'\otimes_R A^{\otimes n})\midotimes_{\mathclap{R'\otimes_R \fixpower{A}{n}{G}}}\,(R'\otimes_R R) \cong A'^{\otimes_{R'} n} \midotimes_{\mathclap{\fixpower[R']{A'}{n}{G}}} R';\]
 the tensor product on the left is $R'\otimes_R(\gclose{A}{n}{\phi})$, and the one on the right is $\gclose[R']{A'}{n}{\phi'}$.
\end{proof}

Note that in the special case $G=\permgroup{n}$, this provides a much simpler proof of \cite[Theorem 1]{Bha14}.

%%%%%% Universal algebras with a closure datum

\begin{remark}\label{universal-closures}
 Let $R$ be a ring, let $A$ an $R$-algebra of rank $n$, let $G$ be a subgroup of $\permgroup{n}$. Let $R'$ be an arbitrary $R$-algebra and $A'= R'\otimes_R A$ the resulting $R'$-algebra of rank $n$, and consider the set of $G$-closure data for $A'$ over $R'$ as $R'$ varies.
 We have the following bijections, natural in $R'$, making this functor representable:
 \begin{gather*}
 \{G\text{-closure data for }A'\text{ over }R'\}\\
 \begin{aligned}
  &\longleftrightarrow\{\fixpower[R']{A'}{n}{\permgroup{n}}\text{-algebra homomorphisms } \fixpower[R']{A'}{n}{G}\to R'\} \\
  &\longleftrightarrow\{R'\text{-algebra homomorphisms }\fixpower[R']{A'}{n}{G}\otimes_{\fixpower[R']{A'}{n}{\permgroup{n}}}R'\to R'\}\\
  &\longleftrightarrow\{R'\text{-algebra homomorphisms }R'\otimes_R(\fixpower{A}{n}{G}\otimes_{\fixpower{A}{n}{\permgroup{n}}}R)\to R'\}\\
  &\longleftrightarrow\{R\text{-algebra homomorphisms }\fixpower{A}{n}{G}\otimes_{\fixpower{A}{n}{\permgroup{n}}}R\to R'\}.
 \end{aligned}
 \end{gather*}
 In particular, setting $R'=\fixpower{A}{n}{G}\otimes_{\fixpower{A}{n}{\permgroup{n}}}R$, we find that $A'$ has a canonical $G$-closure datum corresponding to the identity map on $R'$.
 This $G$-closure datum is universal: every $G$-closure datum for every base extension of $A$ can be obtained via base extension from this $G$-closure datum for $A'$ over $R'$.
 
 The special case of $G=\altgroup{n}$ is particularly nice: the resulting $R$-algebra $\discalg{A/R}=\fixpower{A}{n}{\altgroup{n}}\otimes_{\fixpower{A}{n}{\permgroup{n}}}R$ is the discriminant algebra for $A$ over $R$ defined in \cite{Bie15}, which we can now interpret as the universal $R$-algebra such that base-changing to it gives $A$ an $\altgroup{n}$-closure.
 We explore $\altgroup{n}$-closure data further in \cref{an-closures}.
 
 Letting $R'$ vary again, we find through similar reasoning that if $H\subset G$ are subgroups of $\permgroup{n}$ and $\phi$ is a $G$-closure datum for $A$ over $R$, then $R$-algebra homomorphisms $\fixpower{A}{n}{H}\otimes_{\fixpower{A}{n}{G}}R\to R'$ correspond to the set of $H$-closure data on $A'$ over $R'$ inducing the closure datum $(G,\phi')$ of \cref{g-closure-base-change}.
 
 In particular, we can now give an interpretation to the closure algebra associated with a given closure datum $(G,\phi)$ for $A$ over $R$: the closure algebra
 \[\gclose{A}{n}{\phi} = A^{\otimes n}\otimes_{\fixpower{A}{n}{G}}R\]
 is the universal $R$-algebra for which base changing to it gives $A$ a $1$-closure datum inducing the given $G$-closure datum.
 If we think of a $G$-closure datum as partial factorization information for each characteristic polynomial of elements of $A$---a more precise version of this idea is found in \cref{factorization-data}---then the $G$-closure algebra is the universal algebra over which every characteristic polynomial splits completely in a way respecting this partial information.
\end{remark}

%%%%%%

\section{\'Etale algebras}
\label{section_etale}

%%%%%% Lemma characterizing finite etale algebras

In this section, we will consider closure data for \emph{finite \'etale} algebras, namely locally free algebras for which the trace form $(a,a')\mapsto \trace(aa')$ is non-degenerate.
Examples include the \emph{trivial} \'etale algebras, of the form $R\to R^X:= \prod_{x\in X} R$ for some finite set $X$, as well as finite separable field extensions.
First, we recall a lemma characterizing finite \'etale algebras as those which are \'etale-locally trivial:

\begin{lemma}[{\cite[Lemma 15]{Bha14}}]\label{fin-etale}
 Let $R$ be a ring and $A$ an $R$-algebra finitely generated as an $R$-module.
 Then $A$ is \'etale of rank $n$ if and only if there is an \'etale cover $R\to S$ such that $S\otimes_R A\cong S^n$ as $S$-algebras.
\end{lemma}

(An $R$-algebra $S$ forms an \'etale cover of $R$ if $S$ is \'etale over $R$ and $\spec(S)\to\spec(R)$ is surjective, but we will not need this definition in order to apply \cref{fin-etale}.)

%%%%%% The Ferrand homomorphism for trivial algebras
Thus is it helpful to first consider the Ferrand homomorphism and closure data for trivial algebras:

\begin{lemma}
 \label{ferrand-trivial}
 If $R$ is a ring and $X$ is an $n$-element set, then the Ferrand homomorphism of the trivial rank-$n$ $R$-algebra $R^X$
 \[\Ferrand[R]{R^X}\colon\fixpower{(R^X)}{n}{\permgroup{n}}\cong R^{X^n/\permgroup{n}}\to R\] is projection onto the factor indexed by the orbit of bijections $\Bij(\set{n},X)$.
\end{lemma}
\begin{proof}
 See \cite[Ex. 3.1.3(b)]{Fer98}.
 We could alternatively deduce this result from \cref{1-closure-data}, by choosing an arbitrary bijection $\pi\colon X\to\set{n}$ and pulling back the canonical $1$-closure datum on $R^n$ to $R^X$ and then restricting it to $\fixpower{(R^X)}{n}{\permgroup{n}}$.
\end{proof}

%%%%%% G-closures of trivial algebras are trivial

To understand closure data for finite \'etale algebras, then, it will be helpful to understand closure data for trivial \'etale algebras.

\begin{lemma}
\label{gclose-trivial}
 Let $R$ be a ring, $X$ a finite set of cardinality $n$, and $G$ a subgroup of $\permgroup{n}$.
 Then $G$-closure data for $R^X$ over $R$ correspond bijectively to $R$-algebra homomorphisms $R^I\to R$, where $I$ is the set of $G$-orbits of $\Bij(\set{n},X)$ under the action of $G$ by precomposition.
 Furthermore, every $G$-closure of $R^X$ is isomorphic to $R^{\card{G}}$ as an $R$-algebra.
\end{lemma}

\begin{proof}
 Recall from \cref{universal-closures} that $G$-closure data for $R^X$ over $R$ correspond to $R$-algebra homomorphisms
 \[\fixpower{(R^X)}{n}{G}\ \midotimes_{\mathclap{\fixpower{(R^X)}{n}{\permgroup{n}}}}\ R \to R.\]
 We can write $(R^X)^{\otimes n}$ as $R^{\Map(\set{n},X)}$, with $G$ acting on the set of basis idempotents $\{\basis_f\,|\, f\colon\set{n}\to X\}$ via $\sigma\cdot \basis_f = \basis_{f\circ\sigma^{-1}}$.
 The $G$-invariants, then, have an $R$-basis of idempotents $\basis_O = \sum_{f\in O}\basis_f$ for each $G$-orbit $O$ of $\Map(\set{n},X)$.
 
 By \cref{ferrand-trivial}, the Ferrand homomorphism $R^{\Map(\set{n},X)/\permgroup{n}}\to R$ is projection onto the factor indexed by $\Bij(\set{n},X)$.
 Hence every $\basis_O$ in $R^{\Map(\set{n},X)/G}$ with $O\not\subset\Bij(\set{n},X)$ is sent to zero in the tensor product.
 So $G$-closure data for $R^X$ over $R$ are parametrized by homomorphisms to $R$ from \[R^{\Map(\set{n},X)/G}/\bigl(\basis_O:O\not\subseteq\Bij(\set{n},X)\bigr) = R^{\Bij(\set{n},X)/G}= R^I.\]
 
 Now we show that each $G$-closure algebra of $R^X$ is isomorphic to $R^{\card{G}}$.
 Choose a homomorphism $R^I\to R$; this partitions $\spec(R)$ into $\card{I}$ disjoint affine open subsets on which the map is a projection.
 Then working locally, assume we have the $G$-closure datum corresponding to the projection $R^I\to R$ indexed by the $O$th factor for some $G$-orbit $O\subset\Bij(\set{n},X)$.
 Then the associated closure algebra is
 \[R^{\Map(\set{n},X)} \ \midotimes_{\mathclap{R^{\Map(\set{n},X)/G}}}\ R\quad \cong R^O,\]
 and since $G$ acts freely on $\Bij(\set{n},X)$, we have $\card{O}=\card{G}$ and $R^O\cong R^{\card{G}}$.
 In the general case, we find that $\spec(R)$ is a disjoint union of open subsets on which the closure algebra is trivial; the closure algebra is hence globally trivial as well.
\end{proof}

%%%%%% G-closures of etale algebras are etale

\begin{corollary}
 Let $R$ be a ring and $A$ a rank-$n$ \'etale $R$-algebra.
 Then for every closure datum $(G,\phi)$ for $A$ over $R$, the associated closure algebra $\gclose{A}{n}{\phi}$ is a rank-$\card{G}$ \'etale $R$-algebra.
\end{corollary}

\begin{proof}
 By \cref{fin-etale}, there is an \'etale cover $R\to S$ such that $S\otimes_R A\cong S^n$ as $S$-algebras.
 By \cref{g-closure-base-change}, we obtain a $G$-closure datum $\phi_S$ for $S\otimes_R A$ over $S$, for which the associated closure $\gclose[S]{(S\otimes_R A)}{n}{\phi_S}$ is isomorphic to $S\otimes_R (\gclose{A}{n}{\phi})$.
 But by \cref{gclose-trivial}, the associated $G$-closure of $S\otimes_R A\cong S^n$ is isomorphic to $S^{\card{G}}$.
 Therefore $S\otimes_R (\gclose{A}{n}{\phi})\cong S^{\card{G}}$, so by \cref{fin-etale} again $\gclose{A}{n}{\phi}$ is an \'etale $R$-algebra of rank $\card{G}$.
\end{proof}

%%%%%% Theorem for finite \'etale algebras

In case $R$ is a \emph{connected} ring, that is, $R$ contains exactly two idempotents $0$ and $1$, then for each choice of homomorphism $R\to K$ with $K$ a separably closed field, there is a profinite group $\fundgroup{R}$ called the \emph{\'etale fundamental group} of $R$, and a contravariant equivalence of categories
\[\{\text{finite \'etale $R$-algebras}\}\longleftrightarrow\{\text{finite sets with a continuous $\fundgroup{R}$-action}\}\]
sending an \'etale algebra $A$ to the finite $\fundgroup{R}$-set $\Hom_R(A,K)$; see \cite{Len08} for more details.
In this setting, we have the following interpretation of closure data and closure algebras in terms of the corresponding $\fundgroup{R}$-sets:

\begin{theorem}\label{main-etale}
 Let $R$ be a connected ring with \'etale fundamental group $\fundgroup{R}$.
 Let $A$ be a rank-$n$ \'etale $R$-algebra with corresponding $\fundgroup{R}$-set $X$, and let $G$ be the image of $\fundgroup{R}$ in $\Bij(X,X)$.
 Let $H$ be a subgroup of $\permgroup{n}$.
 Then $H$-closure data for $A$ over $R$ are in one-to-one correspondence with bijections $f:\set{n}\simto X$ such that $f^{-1} G f \subset H$, up to precomposing $f$ by permutations in $H$.
 
 Furthermore, if $B$ is the finite \'etale algebra corresponding to the $\fundgroup{R}$-set $G$, then every $H$-closure of $A$ over $R$ is isomorphic to $B^{\card{H}/\card{G}}$.
\end{theorem}

\begin{proof}
 Recall that $H$-closure data for $A$ over $R$ correspond to homomorphisms
 \[\fixpower{A}{n}{H}\midotimes_{\fixpower{A}{n}{\permgroup{n}}} R \to R,\]
 and thus to $\fundgroup{R}$-equivariant maps
 \[\{\ast\} \to X^n/H \times_{X^n/\permgroup{n}} \{\ast\},\]
 that is, $\fundgroup{R}$-invariant elements of $X^n/H$ whose images in $X^n/\permgroup{n}$ are the class of bijections $\Bij(\set{n},X)$.
 These in turn correspond to the $\fundgroup{R}$-invariant (i.e.\ $G$-invariant, since the action is via $\fundgroup{R}\onto G$) $H$-orbits of $\Bij(\set{n},X)$.
 Write such an $H$-orbit as $fH$ for some bijection $f:\set{n}\simto X$; then the condition that $fH$ be $G$-invariant is the equality $GfH=fH$, or the containment $Gf\subset fH$.
 Thus we may say that $fH$ is a $G$-invariant $H$-orbit if and only if $f^{-1}Gf\subset H$.
 Therefore $H$-closure data correspond to bijections $f$ (up to precomposition by elements of $H$) such that $f^{-1}Gf\subset H$, as desired.
 
 Now given such a $G$-invariant $H$-orbit $O$ of $\Bij(\set{n},X)$, giving a $\fundgroup{R}$-equivariant function $\{\ast\}\to X^n/H$,  we find that the $\fundgroup{R}$-set corresponding to the associated $H$-closure algebra is 
 \[X^n \times_{X^n/H} \{\ast\} = O.\]
 Now the action of $\fundgroup{R}$ on $\Bij(\set{n},X)$ is via $G$, and the $G$-action on $\Bij(\set{n},X)$ is free, so as a $\fundgroup{R}$-set it is isomorphic to a disjoint union of copies of $G$.
 Then so is the $G$-invariant subset $O$, and by comparing cardinalities we find that $O \cong \coprod_{|H|/|G|} G$ as $\fundgroup{R}$-sets.
 Therefore the $H$-closure algebra corresponding to $O$ is isomorphic to $B^{|H|/|G|}$, as claimed.
\end{proof}

\begin{remark}
Note that if $H\subset K$ are subgroups of $\permgroup{n}$, then induction of $H$-closure data to $K$-closure data sends the $H$-orbit of a bijection $f\colon\set{n}\simto X$ to the larger $K$-orbit of $f$.
Then for each bijection $f$, there is a smallest subgroup $G_f\subset\permgroup{n}$ for which $f$ gives a $G_f$-closure datum, namely $G_f=f^{-1} G f$, and this $G_f$-closure datum is therefore minimal.
Two bijections $f,f':\set{n}\simto X$ give the same minimal closure datum if and only if $G_f = G_{f'}$ and $fG_f = f'G_{f'}$, which implies that $Gf = Gf'$, so $f$ and $f'$ are related via postcomposition by an element of $G$.
Thus the minimal closure data for $A$ over $R$ are in bijection with the $G$-orbits of $\Bij(\set{n},X)$.
Since the action of $\permgroup{n}$ on these is transitive, all the minimal closure data are isomorphic, as we claimed in \cref{section_relationships}.
Note also that the closure algebras associated to the minimal closure data are all isomorphic to $B$, the \'etale algebra corresponding to the finite $\fundgroup{R}$-set $G$, as in the usual Galois theory of projective separable ring extensions.
\end{remark}
 
%%%%%% Theorem for finite separable field extensions
%
%Our last theorem specializes to the case of a finite separable field extension $K\into L$, for which we recall from the introduction that there is a canonical closure datum $(G,\phi_{L/K})$ where $G$ is the Galois group of $L$'s normal closure.
%We now show that this is the unique minimal closure datum for such an extension, up to isomorphism:
%
%\begin{theorem}
% Let $K\into L$ be a separable degree-$n$ field extension, with $N$ a Galois closure of $L$ over $K$.
% Let $G=\Gal(N/K)$ be the resulting Galois group, and choose an identification of the $n$-element set $\Hom_K(L,N)$ with $\set{n}$, so that $G$ is identified with a subgroup of $\set{n}$.
% Then every closure datum for $L/K$ is induced by one isomorphic to $(G,\phi_{L/K})$.
% 
% Furthermore, if $(H,\phi)$ is such a closure datum, then the $H$-closure $\gclose{L}{n}{\phi}$ is isomorphic to $N^{\card{H}/\card{G}}$.
%\end{theorem}

%%%%%%

\section{Product algebras}
\label{section_products}

In this section, we consider the closure data that arise on product algebras $A_1\times\dots\times A_k$ given closure data on each $A_i$.
%%%%%% Main theorem of product algebras
Our first main theorem is as follows:

\begin{theorem}\label{main-products}
 Let $R$ be a ring, and let $A_i$ be an $R$-algebra of rank $n_i$ for each $i\in\set{k}$, each with a closure datum $(G_i, \phi_i)$.
 Set
 \begin{itemize}
  \item $n := \sum_{i=1}^k n_i$.
  \item $A := \prod_{i=1}^k A_i$, an $R$-algebra of rank $n$.
  \item $G := \prod_{i=1}^k G_i$, considered as a subgroup of $\permgroup{n}$ via the action of each $G_i$ on the $n_i$-element set $\{n_1+\dots+n_{i-1}+1,\dots,\,n_1+\dots + n_{i-1}+n_i\}$.
  \item $\phi\colon \fixpower{A}{n}{G}\to R$ equal to the composite
  \end{itemize}
  \[\begin{tikzcd}
  \ds\fixpower{A}{n}{G}\cong\bigotimes_{i=1}^k \fixpower{A}{n_i}{G_i}\arrow[two heads]{r} & \ds\bigotimes_{i=1}^k \fixpower{A_i}{n_i}{G_i}\arrow{rr}{\bigotimes_{i=1}^k \phi_i} & & R
  \end{tikzcd}.\]
 Then $(G,\phi)$ is a closure datum for $A$ over $R$.
 
 Furthermore, the $R$-algebra homomorphism $A^{\otimes n}\cong \bigotimes_{i=1}^k A^{\otimes n_i} \to \bigotimes_{i=1}^k A_i^{\otimes n_i}$ descends to an isomorphism $\gclose{A}{n}{\phi}\cong \bigotimes_{i=1}^k (\gclose{A_i}{n_i}{\phi_i})$.
\end{theorem}

\begin{proof}
 First we show that $\phi$ restricts to the Ferrand homomorphism $\fixpower{A}{n}{\permgroup{n}}\to R$.
 Let $a = (a_1,\dots,a_k)\in A$, and consider the image of $(x-a)^{\otimes n}$ under $\phi\otimes\id_{R[x]}$.
 We find 
 \[(x-a)^{\otimes n} = \bigotimes_{i=1}^k (x-a)^{\otimes n_i} \longmapsto \bigotimes_{i=1}^k (x-a_i)^{\otimes n_i} \longmapsto \prod_{i=1}^k \charpoly{a_i}(x) = \charpoly{a}(x),\]
 so looking at each coefficient of $x^{n-k}$, we have that $\phi$ sends $e_k(a)$ to $s_k(a)$ as desired.
 (That the characteristic polynomial $\charpoly{a}(x)$ of $a\in A$ factors as the product of each characteristic polynomial $\charpoly{a_i}(x)$ of $a_i\in A_i$ is easy to check locally when each $A_i$ has an $R$-basis: then $a=(a_1,\dots,a_k)$ acts block diagonally.)
 
 Next we show that the closure algebra associated to $(G,\phi)$ is the tensor product of all the $\gclose{A_i}{n_i}{\phi_i}$.
 Note that 
 \[
  \fixpower{A}{G}{n} = (A^{\otimes n_1}\otimes \dots \otimes A^{\otimes n_k})^{G_1\times\dots\times G_k} \cong \fixpower{A}{n_1}{G_1}\otimes\dots\otimes \fixpower{A}{n_k}{G_k},
 \] 
 since the natural map is easily checked to be an isomorphism whenever $A$ is a free $R$-module.
 So we obtain
 \[
  A^{\otimes n} \midotimes_{\fixpower{A}{n}{G}} R \ \cong\ \Bigl(\bigotimes_{i=1}^k A^{\otimes n_i}\Bigr) \midotimes_{\left(\bigotimes_{i=1}^k \fixpower{A}{n_i}{G_i}\right)} R \ \cong\  
  \bigotimes_{i=1}^k\left(A^{\otimes n_i}\ \midotimes_{\mathclap{\fixpower{A}{n_i}{G_i}}}\ R\right).
 \]
 Then all we must show is that $A^{\otimes n_i}\otimes_{\fixpower{A}{n_i}{G_i}}R$ is isomorphic to $\gclose{A_i}{n_i}{\phi_i}= A_i^{\otimes n_i}\otimes_{\fixpower{A_i}{n_i}{G_i}}R$.
 Indeed, if we let $\basis_i\in A$ be the element $(0,\dots,0,1,0,\dots,0)$ with a $1$ in the $i$th place, then $\basis_i\otimes \dots \otimes \basis_i\in \fixpower{A}{n_i}{G_i}$ is sent to $1$ in $R$. 
 Hence for each element $a\in I:= \{(a_1,\dots,a_n)\in A:a_i=0\}$, the image of $\conjugate{a}{j}$ in $A^{\otimes n_i}\otimes_{\fixpower{A}{n_i}{G_i}}R$ is equal to that of $\conjugate{a}{j}\cdot (\basis_i\otimes\dots\otimes\basis_i) = 0$.
 Therefore 
 \[
  A^{\otimes n_i}\ \midotimes_{\mathclap{\fixpower{A}{n_i}{G_i}}}\ R \ \cong\  (A/I)^{\otimes n_i}\ \midotimes_{\mathclap{\fixpower{A}{n_i}{G_i}}}\ R \ \cong\  A_i^{\otimes n_i}\ \midotimes_{\mathclap{\fixpower{A}{n_i}{G_i}}}\ R.
 \]
 Last, observe that the two maps defining the tensor product
 \[
 \fixpower{A}{n_i}{G_i}\to A^{\otimes n_i}\to A_i^{\otimes n_i}\text{ and }
 \fixpower{A}{n_i}{G_i}\to \fixpower{A_i}{n_i}{G_i}\to R\]
 both factor through $\fixpower{A_i}{n_i}{G_i}$; we may therefore substitute it in the base of the tensor product:
 \[
 A_i^{\otimes n_i}\ \midotimes_{\mathclap{\fixpower{A}{n_i}{G_i}}}\ R \ \cong\  A_i^{\otimes n_i}\ \midotimes_{\mathclap{\fixpower{A_i}{n_i}{G_i}}}\ R \ =\  \gclose{A_i}{n_i}{\phi_i}.
 \]
 Thus $\gclose{A}{n}{\phi} = \bigotimes_{i=1}^k(\gclose{A_i}{n_i}{\phi_i})$ as desired.
\end{proof}

%%%%%% Note that universally norm-preserving maps to products give intransitive extensions.

 Note that \cref{main-products} implies that a universally norm-preserving homomorphism from a rank-$n$ algebra $A$ to a product of lower rank algebras $\prod_{i=1}^k B_i$ (with each $B_i$ of rank $n_i$) produces a $\prod_{i=1}^k\permgroup{n_i}$-closure datum for $A$.
 We might ask whether all $\prod_{i=1}^k\permgroup{n_i}$-closure data arise in this way.
 \Cref{factorization-data} shows that the answer is yes if $A$ is monogenic.
 We can also see elementarily that $1 = \permgroup{1}\times\dots\times\permgroup{1}$-closures arise in this way via the following proposition:
 
%%%%%% Fact about 1-closures and universally norm-preserving maps

\begin{proposition}\label{1-closure-data}
 Let $R$ be a ring and $A$ an $R$-algebra of rank $n$.
 Let $f_1,\dots,f_n$ be $R$-algebra homomorphisms from $A$ to $R$; we can compile these into two single $R$-algebra homomorphisms 
 \begin{align*}
  \bigotimes_{i=1}^n f_i&\colon A^{\otimes n}\to R:\conjugate{a}{i}\mapsto f_i(a)\text{ for each }i\in\set{n}\\
  \prod_{i=1}^n f_i&\colon A\to R^n:a\mapsto (f_1(a),\dots,f_n(a))
 \end{align*}
 Then $\bigotimes_i f_i$ is a $1$-closure datum for $A$ over $R$ if and only if $\prod_i f_i$ is a universally norm-preserving homomorphism from $A$ to $R^n$.
 \end{proposition}

\begin{proof}
 First, note that the characteristic polynomial for an element $(r_1,\dots,r_n)$ of $R^n$ is $\prod_{i=1}^n (x-r_i)$, so the $s_k$ of this tuple is the $k$th elementary symmetric polynomial in the $r_i$.
 Then $\prod_i f_i$ being universally norm-preserving is equivalent to the $k$th elementary symmetric polynomial in the $f_i(a)$ always equaling $s_k(a)$.
 But $\bigotimes_i f_i$ sends $e_k(a)$ to the $k$th elementary symmetric polynomial in the $f_i(a)$, so $\prod_i f_i$ being universally norm-preserving is equivalent to $\bigotimes_i f_i$ sending each $e_k(a)$ to $s_k(a)$.
 This is in turn equivalent to $\bigotimes_i f_i$ restricting to $\Ferrand{A}$ on $\fixpower{A}{n}{\permgroup{n}}$ by \cref{ferrand-uniqueness}.
 \end{proof}
 
 \begin{remark}
  In particular, the $n$ projections $\pi_1,\dots,\pi_n\colon R^n\to R$ give a canonical $1$-closure datum for $R^n$ over $R$, thereby inducing a canonical $G$-closure datum for each subgroup $G\subset\permgroup{n}$.
 Furthermore, if $\prod_i f_i\colon A\to R^n$ is universally norm-preserving, then $\bigotimes_i f_i$ is the $1$-closure datum for $A$ obtained via $\prod_i f_i$ from the canonical $1$-closure datum for $R^n$.
 Then given any $G$-closure datum $\phi$ for $A$ over $R$, we can interpret the $G$-closure algebra $\gclose{A}{n}{\phi}$ as the universal $R$-algebra $R'$ such that $A':= R'\otimes_R A$ gains a universally norm-preserving homomorphism to $R'^n$ for which the base change of $\phi$ is the pullback of the canonical $G$-closure datum on $R'^n$.
 %Interpretation in Galois case as 
\end{remark}

%%%%%% Theorem that for larger G we get a power.
Our second main theorem shows that if we replace $G$ in \cref{main-products} with a larger group $H$ (while keeping $H\cap\permgroup{n_i}$ equal to $G_i$) then the induced closure algebra is just a power of the $G$-closure algebra, generalizing \cite[Theorem 6]{Bha14} that 
\[\gclose{A}{n}{\Ferrand{A}}\cong \Bigl(\bigotimes_{i=1}^k \gclose{A_i}{n_i}{\Ferrand{A_i}}\Bigr)^{\binom{n}{n_1,n_2,\dots,n_k}}.\]

\begin{theorem}\label{stronger-products}
 In the setting of \cref{main-products}, let $H\subset\permgroup{n}$ be a subgroup such that $H\cap \permgroup{n_i} = G_i$, where we regard $\permgroup{n_i}$ as a subgroup of $\permgroup{n}$ via its action on $\{n_1+\dots+n_{i-1}+1,\dots,\,n_1+\dots + n_{i-1}+n_i\}$.
 Then $H\supset G$, and the induced $H$-closure datum $\phi|_{\fixpower{A}{n}{H}}$ has associated closure algebra
 \begin{align*}
 \gclose{A}{n}{(\phi|_{\fixpower{A}{n}{H}})} &\cong \left(\gclose{A}{n}{\phi}\right)^{(H:G)}\cong \Bigl(\bigotimes_{i=1}^k \gclose{A_i}{n_i}{\phi_i}\Bigr)^{(H:G)}.
 \end{align*}
\end{theorem}

\begin{proof}
 Again, for $i\in\set{k}$ let $\basis_i\in A=\prod_{i=1}^k A_i$ be the idempotent $(0,\dots,0,1,0,\dots,0)$ with a $1$ in the $i$th place.
 Let $e$ be the idempotent of $A^{\otimes n}$ given by 
 \[e = \basis_1\otimes\dots\otimes \basis_1\otimes \basis_2\otimes\dots\otimes \basis_2 \otimes\dots\otimes\basis_k\otimes\dots\otimes\basis_k,\]
 with $n_i$ tensor factors of each $\basis_i$.
 Then $e$ is $G$-invariant, so the $H$-orbit $\{h.e : h\in H\}$ of $e$ will be in natural bijection with the set $H/G$ of left cosets of $G$ in $H$.
 Let $\tilde e$ be the sum of all the elements of this orbit; then $\tilde e$ is $H$-invariant and sent to $1$ under $\phi|_{\fixpower{A}{n}{H}}$.
 
 Now the $\card{H/G}$ idempotents $\{h.e : h\in H\}$ map to idempotents of the $H$-closure $\gclose{A}{n}{\phi|_{\fixpower{A}{n}{H}}}$ that are permuted transitively by the action of $H$, and moreover these idempotents are orthogonal and have sum $1$.
 Therefore $\gclose{A}{n}{\phi|_{\fixpower{A}{n}{H}}}$ splits as a product of $\card{H/G}$ isomorphic factors.
 We claim that the factor corresponding to the idempotent $e$ is canonically isomorphic to $\gclose{A}{n}{\phi}$.
 Indeed, we are comparing the two quotients
 \[A^{\otimes n}/(e-1, y - \phi(y): y\in \fixpower{A}{n}{H})\text{ and }A^{\otimes n}/(x - \phi(x): x\in \fixpower{A}{n}{G})\]
 The right-hand ideal clearly contains the left-hand ideal; we show conversely that everything in the right-hand ideal is already zero in the left-hand quotient.
 Let $x\in\fixpower{A}{n}{G}$.
 By working locally, we may assume that the $A_i$ are all free, and by expanding $x$, we may assume $x$ is a sum over the $G$-orbit of a pure tensor with each tensor factor a basis element of some $A_i$.
 Let $y$ be the corresponding $H$-orbit sum.
 There are two cases, according as $x\cdot e = x$ (when, in order, the tensor factors of $x$ consist of $n_1$ from $A_1$, $n_2$ from $A_2$, etc.) or $x\cdot e = 0$.
 In the former case, we have $y\cdot e = x$ as well, since $e$ annihilates every term of $y-x$.  
 Then $\phi(x) = \phi(y\cdot e) = \phi(y)\phi(e) = \phi(y)$, so in $\bigl(\gclose{A}{n}{\phi|_{\fixpower{A}{n}{H}}}\bigl)/(e-1)$ we have
 \[\phi(x) = \phi(y) = y = y\cdot 1= y\cdot e = x.\]
 In the case that $x\cdot e = 0$, then $\phi(x) = \phi(x)\cdot 1 = \phi(x)\phi(e) = \phi(x\cdot e) = 0$.  And in $\gclose{A}{n}{\phi|_{\fixpower{A}{n}{H}}}/(e-1)$ we thus have 
 \[\phi(x) = 0 = x\cdot e = x\cdot 1 = x. \]
 Therefore the ideal $(x-\phi(x) : x\in \fixpower{A}{n}{G})$ is equal to $(y-\phi(y) : y\in\fixpower{A}{n}{H}) + (e-1)$ as claimed.
 So $\gclose{A}{n}{\phi|_{\fixpower{A}{n}{H}}}$ is isomorphic to a product of $\card{H/G}$ copies of $\gclose{A}{n}{\phi}$.
\end{proof}

%%%%%%

\section{$\altgroup{n}$-closure data}
\label{an-closures}

%%%%%% A_n-closures generally

Recall from \cref{universal-closures} that if $A$ is a rank-$n$ algebra over $R$, then $\altgroup{n}$-closure data for $A$ over $R$ correspond to $R$-algebra homomorphisms to $R$ from the \emph{discriminant algebra}
\[\discalg{A/R}:= \fixpower{A}{n}{\altgroup{n}} \ \midotimes_{\mathclap{\fixpower{A}{n}{\permgroup{n}}}}\ R.\]
The discriminant algebra is always a rank-$2$ algebra over the base ring, and furthermore there is a canonical isomorphism
\[\extpower^n A \simto \extpower^2 \discalg{A/R}\]
sending $a_1\wedge\dots\wedge a_n$ to $1\wedge \dot\gamma(a_1,\dots,a_n)$, where $\dot\gamma(a_1,\dots,a_n)$ is the image in $\discalg{A/R}$ of the $\altgroup{n}$-invariant element $\gamma(a_1,\dots,a_n) = \sum_{\sigma\in\altgroup{n}}a_{\sigma(1)}\otimes\dots\otimes a_{\sigma(n)}$ of $\fixpower{A}{n}{\altgroup{n}}$.
This isomorphism respects the discriminant bilinear forms on $\extpower^n A$ and $\extpower^2\discalg{A/R}$; see \cite[Theorem 4.1]{Bie15} for proofs of all the above statements.

%%%%%% A_n-closures for algebras free as modules

If $A$ is not merely locally free as an $R$-module, but free with $R$-basis $(\theta_1,\dots,\theta_n)$, then $\extpower^n A$ is free with generator $\theta_1\wedge\dots\wedge \theta_n$.
Therefore $\extpower^2\discalg{A}$ is free with generator $1\wedge\dot\gamma(\theta_1,\dots,\theta_n)$, and hence $\discalg{A}$ itself has $R$-basis $(1,\dot\gamma(\theta_1,\dots,\theta_n))$.
We can thus present $\discalg{A/R}$ abstractly as an $R$-algebra $R[y]/(q(y))$, where $q$ is a monic quadratic polynomial, the characteristic polynomial of $\dot\gamma(\theta_1,\dots,\theta_n)$ in $\discalg{A/R}$. 
Futhermore, the discriminant of $q$ is the same as the discriminant of $A$ with respect to its basis $(\theta_1,\dots,\theta_n)$.
So $\altgroup{n}$-closures of $A$ over $R$ correspond to $R$-algebra homomorphisms $R[y]/(q(y))\to R$, i.e.\ roots of $q$ in $R$.

%%%%%% Example: discalg of x^3 + ax + b

\begin{example}\label{monogenic-A3-example}
 Let $R$ be a ring with elements $a,b\in R$, and let $A$ be the $R$-algebra $R[x]/(x^3 + ax + b)$ with $R$-basis $(1,x,x^2)$.
 Then $\discalg{A/R}$ has $R$-basis $(1, \dot\gamma(1,x,x^2))$.
 Since the sum and product of $\dot\gamma(1,x,x^2)$ and $\dot\gamma(1,x^2,x)$ are both in $R$, we have a monic quadratic polynomial of which $\dot\gamma(1,x,x^2)$ is a root, and which must therefore be its characteristic polynomial:
 \[y^2 - \bigl(\dot\gamma(1,x,x^2) + \dot\gamma(1,x^2,x)\bigr)y + \bigl(\dot\gamma(1,x,x^2) + \dot\gamma(1,x^2,x)\bigr).\]
 We can compute these coefficients with the Ferrand homomorphism $\fixpower{A}{3}{\permgroup{3}}\to R$:
 \begin{align*}
 \dot\gamma(1,x,x^2) + \dot\gamma(1,x^2,x) &= \Ferrand{A/R}(\gamma(1,x,x^2) + \gamma(1,x^2,x)) = 3b\\
 \dot\gamma(1,x,x^2) \dot\gamma(1,x^2,x) &= \Ferrand{A/R}(\gamma(1,x,x^2)\gamma(1,x^2,x)) = a^3 + 9b^2.
 \end{align*}
 (See \cite[Example 5.6]{Bie15} for the full and more general computation.)
 Then $\discalg{A/R}\cong R[y]/(y^2 - (3b)y + (a^3 + 9b^2))$, so $\altgroup{3}$-closure data for $A$ over $R$ correspond to roots of $y^2 - (3b)y+(a^3 + 9b^2)$ in $R$.
 In particular, if an $\altgroup{3}$-closure datum exists then the discriminant $(3b)^2 - 4(a^3 + 9b^2) = -4a^3 - 27b^2$ is a square in $R$.
\end{example}

%%%%%% Example: F4 over F2.

\begin{example}
 Consider the degree-$2$ separable extension $\F{4}$ over $\F{2}$.
 Since any quadratic algebra is canonically isomorphic to its discriminant algebra (\cite[Proposition 5.1]{Bie15}), there exists an $\altgroup{2}=1$-closure datum if and only if there is a map $\F{4}\to\F{2}$, which there is not, even though the discriminant of $\F{4}$ over $\F{2}$ is $1$, a square.
 This is consistent with the Galois group of $\F{4}$ over $\F{2}$ being $\permgroup{2}$.
 
 On the other hand, the cubic $\F{2}$-algebra $\F{8}\cong \F{2}[x]/(x^3 + x + 1)$ has discriminant algebra $\F{2}[y]/(y^2-(3\cdot 1)y+(1^3 + 9\cdot 1^2)) = \F{2}[y]/(y^2-y)$, which does admit a map to $\F{2}$.
 Therefore $\F{8}$ has an $\altgroup{3}=\cyclicgroup{3}$-closure datum, which is consistent with having Galois group $\cyclicgroup{3}$.
\end{example}

%%%%%% Discussion that this is a better version of "adjoin sqr root of disc".

Note that this criterion for $A$ to have an $\altgroup{n}$-closure datum, namely that there is an $R$-algebra homomorphism $\discalg{A/R}\to R$, works equally well in every characteristic.
In this respect, the discriminant algebra is a better quadratic resolvent than testing whether the discriminant is a square, which for field extensions only works in characteristic other than $2$.
However, the square-discriminant test does work in a slightly larger generality: when $2$ is a \emph{primoid non-zerodivisor}.

%%%%%% primoid elements

\begin{definition}
 Let $p$ be an element of a ring $R$.
 We say that $p$ is \emph{primoid} if whenever $p^2$ divides a product $ab$, then $p$ divides $a$ or $b$.
\end{definition}

For example, units and prime elements are primoid.
More generally, every power of a prime non-zerodivisor is primoid.
The utility of this notion is that the quadratic formula works over a ring $R$ if $2\in R$ is a primoid non-zerodivisor:

%Units are clearly primoid; if $p$ is a non-unit then $p$ being primoid is weaker than $(p)$ being prime and stronger than $(p)$ being primary.

\begin{lemma}
 Let $R$ be a ring, and let $x\in R$ be a solution to the equation $x^2+bx+c=0$ for fixed $b,c\in R$.
 Then $2x+b$ is a square root of the equation's discriminant $b^2-4c$.
 If $2$ is a primoid non-zerodivisor in $R$, then this assignment $x\mapsto 2x+b$ forms a one-to-one correspondence between the solutions to $x^2+bx+c=0$ and the square roots of the discriminant.
\end{lemma}

\begin{proof}
 That $2x+b$ is a square root of the discriminant is straightforward: $(2x+b)^2 = 4x^2 + 4bx + b^2 = 4(-bx-c) + 4bx + b^2 = b^2-4c$.
 Conversely, suppose that $2$ is a primoid non-zerodivisor and that $d$ is a square root of $b^2-4c$.
 We show that $d$ can be uniquely written as $2x+b$ for some solution $x$ to $x^2+bx+c=0$.
 Consider that $(d+b)(d-b) = d^2-b^2 = -4c$, so since $2$ is primoid we must have $2|(d+b)$ or $2|(d-b)$.
 Then we conclude that since the difference between $d+b$ and $d-b$ is a multiple of $2$, both are multiples of $2$.
 In particular, $d-b$ can be written uniquely as $2x$ for some $x$, since $2$ is a non-zerodivisor.
 And for that $x$, we have $4(x^2+bx+c) = (2x)^2 + 2b(2x)+4c = (d-b)^2 + 2b(d-b) + 4c = d^2 - b^2 + 4c = 0$, so $x^2+bx+c=0$ as desired.
\end{proof}

%%%%%% Discriminant test for free algebras when 2 is primoid

\begin{theorem}\label{sqrt-disc}
 Let $R$ be a ring in which $2$ is a primoid non-zerodivisor, and let $A$ be an $R$-algebra equipped with an $R$-module basis.
 Then $\altgroup{n}$-closure data for $A$ over $R$ correspond to square roots in $R$ of the discriminant of $A$ with respect to that basis.
\end{theorem}

\begin{proof}
 We know that $\altgroup{n}$-closure data for $A$ correspond to roots in $R$ of a quadratic polynomial whose discriminant equals that of $A$ with respect to the given basis.
 But since $2$ is a primoid non-zerodivisor, roots of such a quadratic polynomial correspond to square roots of its discriminant, and thus square roots of the discriminant of $A$.
\end{proof}

%%%%%% 2 a non-primoid counterexample

\begin{example}
 To see that the primoid hypothesis in \cref{sqrt-disc} is necessary, consider the ring $R=\Z[\sqrt{5}]$ and $A = R[x]/(x^2-x-1)$.
 The discriminant of $A$ over $R$ is $(-1)^2 - 4(1)(-1) = 5$, a square in $R$.
 But $A$ does not have an $\altgroup{2}=1$-closure as it does not admit a homomorphism to $R$; the golden ratio is not a $\Z$-linear combination of $1$ and $\sqrt{5}$.
 This is because $2$ is not primoid in $R$: we have $(1+\sqrt{5})(1-\sqrt{5}) = -4$, a multiple of $2^2$, but neither factor is a multiple of $2$.
\end{example}

%%%%%%

\section{Monogenic algebras}
\label{section_monogenic}

%%%%%% Definition of monogenic

\begin{definition}
 Let $R$ be a ring and $A$ an $R$-algebra. 
 We say that $A$ is \emph{monogenic of rank $n$} if there exists an isomorphism $A\cong R[x]/(f(x))$ for some monic degree-$n$ polynomial $f(x)$.
 In particular, a monogenic rank-$n$ $R$-algebra is necessarily \emph{free} of rank $n$ as an $R$-module.
\end{definition}

%%%%%% Two meanings of monogenic are the same.

\begin{remark}
 There is also a weaker notion of ``monogenic,'' meaning just that $A$ is generated by a single element as an $R$-algebra, but if $A$ is a rank-$n$ $R$-algebra then these two notions are equivalent.
 Indeed, suppose $A$ has rank $n$ and is generated as an $R$-algebra by a single element $a$.
 Then we have a surjective $R$-algebra homomorphism $R[x]\onto A$ sending $x\mapsto a$, and since $a$ is a root of its characteristic polynomial $\charpoly{a}(x)$, this map descends to a surjection $R[x]/(\charpoly{a}(x))\onto A$.
 Locally, this is a surjective homomorphism of free rank-$n$ modules, hence must be (locally and globally) an isomorphism.
 So $R[x]/(\charpoly{a}(x))\cong A$.
\end{remark}

%%%%%% Characteristic maps from monogenic algebras are jointly surjective.

\begin{remark}
 Note that given any element $a\in A$, not necessarily a generator, we still obtain a well-defined algebra homomorphism $R[x]/(\charpoly{a}(x))\to A$.
 By \cite[Example 7.2]{Bie15}, this homomorphism is universally norm-preserving.
 Therefore $G$-closure data for $A$ pull back to $G$-closure data for $R[x]/(\charpoly{a}(x))$.
 This may be viewed as a kind of obstruction to the existence of $G$-closure data for $A$ over $R$; there cannot be any unless $R[x]/(\charpoly{a}(x))$ has one too for every $a\in A$.
 For this reason, criteria for monogenic algebras to admit $G$-closure data are useful even if one is interested in algebras that are not necessarily monogenic.
\end{remark}

\subsection{Intransitive closure data}

In this section we consider closure data for a monogenic algebra when the subgroup of $\permgroup{n}$ is of the form $\permgroup{n_1}\times\dots\times \permgroup{n_k}$, with each $\permgroup{n_i}$ acting on $\set{n}$ by permuting the $n_i$ elements
\[\{n_1+\dots+n_{i-1} + 1,\dots,n_1+\dots+n_{i-1} + n_i\}\]
as in \cref{main-products}.
We find that such closure data correspond to factorizations of the defining polynomial of the monogenic algebra.

%%%%%% Intransitive closure data
 
\begin{theorem}\label{factorization-data}
 Let $f(x)$ be a monic degree-$n$ polynomial with coefficients in a ring $R$, and let $n_1, n_2,\dots,n_k$ be natural numbers whose sum is $n$.  
 Then $\permgroup{n_1}\times\dots\times \permgroup{n_k}$-closure data of $A=R[x]/(f(x))$ correspond to factorizations of $f$ into monic factors $f_1(x),\dots, f_k(x)$ of degrees $n_1,\dots,n_k$, resepctively.
 
 Given such a factorization $f(x) = f_1(x)\dots f_k(x)$, set $A_i=R[x]/(f_i(x))$.
 Then the $\permgroup{n_1}\times\dots\times\permgroup{n_k}$-closure algebra associated to this factorization is isomorphic to $\bigotimes_{i=1}^k A_i^{\otimes n_i}/\Ferrand{A_i}$.
\end{theorem}

\begin{proof}
 To produce a $\prod_i\permgroup{n_i}$-closure datum from a factorization, consider the $R$-algebra homomorphism $A\to \prod_{i=1}^k A_i: x\mapsto (x,\dots,x)$.
 It is a universally norm-preserving homomorphism because the characteristic polynomial of $(x,\dots,x)$ in $\prod_{i=1}^k A_i$ is $f_1(x)\dots f_k(x)=f(x)$, which is the characteristic polynomial of $x$ in $A$.
 Therefore the $\prod_i\permgroup{n_i}$-closure datum on $\prod_{i=1}^k A_i$ from \cref{main-products} pulls back to a $\prod_i\permgroup{n_i}$-closure datum on $A$.
 
 Now we must show that every $\prod_i\permgroup{n_i}$-closure datum on $A$ arises in this way.
 Given a homomorphism $\phi:\fixpower{A}{n}{\prod_i\permgroup{n_i}}\to R$ restricting to the Ferrand homomorphism, consider for each $i\in\set{k}$ the image under $\phi\otimes\id_{R[\lambda]}$ of the $\prod_i\permgroup{n_i}$-invariant element
 \[1^{\otimes n_1} \otimes \dots \otimes 1^{\otimes n_{i-1}} \otimes (\lambda - x)^{\otimes n_i}\otimes 1^{\otimes n_{i+1}} \otimes \dots \otimes 1^{\otimes n_k};\]
 this is a monic degree-$n_i$ polynomial in $R[\lambda]$ which we denote by $f_i(\lambda)$.
 Then because the product of these $k$ invariant elements is $(\lambda-x)^{\otimes n}$ and is sent to $f(\lambda)$, we therefore have a factorization $f(\lambda) = \prod_{i=1}^k f_i(\lambda)$.
 To see that this factorization gives rise to the closure datum $\phi:\fixpower{A}{n}{\prod_i \permgroup{n_i}} \cong \bigotimes_i \fixpower{A}{n_i}{\permgroup{n_i}}\to R$, we show that $\phi$'s $i$th component $\fixpower{A}{n_i}{\permgroup{n_i}}\to R$ factors via $\fixpower{A_i}{n_i}{\permgroup{n_i}}$.
 By the fundamental theorem of elementary symmetric polynomials, it is sufficient to check the images of each element of the form $e_\ell(x)\in\fixpower{A}{n_i}{\permgroup{n_i}}$ for $\ell\in\set{n_i}$.
 We may check these simultaneously by adjoining an auxiliary indeterminate $\lambda$ and considering the single element $(\lambda-x)^{\otimes n_i} = \sum_{\ell=1}^{n_i}(-1)^\ell \lambda^\ell e_\ell(x)$.
 Applying the homomorphism $\fixpower{A}{n_i}{\permgroup{n_i}}\onto \fixpower{A_i}{n_i}{\permgroup{n_i}}\to R$ coefficientwise, this element is sent to the characteristic polynomial of $x$ in $A_i$, namely $f_i(\lambda)$.
 But this is equal, by definition of $f_i$, to the image of $(\lambda-x)^{\otimes n_i}$ under $\phi$'s $i$th component $\fixpower{A}{n_i}{\permgroup{n_i}}\to R$.
 Thus we recover $\phi$ as the $\prod_i\permgroup{n_i}$-closure datum corresponding to the factorization $f = \prod_i f_i$.
 
 Now we check that given such a factorization $f(x) = f_1(x)\dots f_k(x)$, the corresponding $\prod_i\permgroup{n_i}$-closure algebra is isomorphic to $\bigotimes_i A_i^{\otimes n_i} / \Ferrand{A_i}$.
Let $\phi$ be the associated closure datum 
\[\fixpower{A}{n}{\prod_i \permgroup{n_i}} \cong \bigotimes_{i=1}^k \fixpower{A}{n_i}{\permgroup{n_i}}\onto \bigotimes_{i=1}^k \fixpower{A_i}{n_i}{\permgroup{n_i}}\to R.\]
We have the isomorphism
 \begin{align*}
  A^{\otimes n}/\Ferrand{A} = A^{\otimes n} \ \midotimes_{\mathclap{\fixpower{A}{n}{\prod_i \permgroup{n_i}}}}\  R \cong \bigotimes_{i=1}^k A^{\otimes n_i} \ \midotimes_{\mathclap{\fixpower{A}{n_i}{\permgroup{n_i}}}}\ R,
 \end{align*}
 where for each $i\in\set{k}$ the map $\fixpower{A}{n_i}{\permgroup{n_i}}\to R$ is the composite
 \[\begin{tikzcd}\fixpower{A}{n_i}{\permgroup{n_i}}\onto \fixpower{A_i}{n_i}{\permgroup{n_i}}\arrow{r}{\Ferrand{A_i}} & R\end{tikzcd}.\]
 Note that this homomorphism tensored with $R[\lambda]$ sends the element
 \[\prod_{j=1}^{n_i} (\lambda - \conjugate{x}{j}) = (\lambda - x) \otimes \dots \otimes (\lambda - x) \mapsto \norm{A_i[\lambda]}(\lambda - x) = f_i(\lambda),\]
 and thus in $A^{\otimes n_i}\otimes_{\fixpower{A}{n_i}{\permgroup{n_i}}} R$ we find that each $f_i(\conjugate xj) = \prod_{j'}(\conjugate xj - \conjugate x{j'}) = 0$.
 Therefore we have
 \[A^{\otimes n_i} \ \midotimes_{\mathclap{\fixpower{A}{n_i}{\permgroup{n_i}}}}\ R \cong A_i^{\otimes n_i} \ \midotimes_{\mathclap{\fixpower{A}{n_i}{\permgroup{n_i}}}}\ R.\]
 Now the two homomorphisms from $\fixpower{A}{n_i}{\permgroup{n_i}}$ in the tensor product both factor through its quotient $\fixpower{A_i}{n_i}{\permgroup{n_i}}$, so we obtain
 \[A_i^{\otimes n_i} \ \midotimes_{\mathclap{\fixpower{A}{n_i}{\permgroup{n_i}}}}\ R \cong A_i^{\otimes n_i} \ \midotimes_{\mathclap{\fixpower{A_i}{n_i}{\permgroup{n_i}}}}\ R = A_i^{\otimes n_i}/\Ferrand{A_i}.\]
 Thus $A^{\otimes n}/\phi \cong \bigotimes_{i=1}^k A_i^{\otimes n_i}/\Ferrand{A_i}$ as desired.
\end{proof}

\begin{corollary}
 If $f(x)\in R[x]$ is irreducible, then every $G\subset\permgroup{n}$ for which $R[x]/(f(x))$ has a $G$-closure datum acts transitively on $\set{n}$.
\end{corollary}

%%%%%% Uniqueness of minimal closures for free quadratic algebras.

\begin{remark}\label{domain-unique}
 Recall the question of whether the minimal closure data for a given algebra are isomorphic.
 This holds for all free quadratic $R$-algebras if and only if $R$ is a domain, as we claimed in \cref{minimal-uniqueness}.
 Namely, suppose that $R$ is a domain and $A$ is a free quadratic $R$-algebra.
 Then $A/R\cong \extpower^2 A$ is also free, so we can choose a basis for $A$ of the form $\{1,a\}$, so that $A\cong R[x]/(x^2-bx+c)$ for some $b,c\in R$.
 If the $\permgroup{2}$-closure datum of $A$ is minimal, then it is the unique closure datum for $A$ over $R$, so the minimal closure data are trivially isomorphic.
 Otherwise, there is a $1 = \permgroup{1}\times\permgroup{1}$-closure datum for $A$ over $R$, corresponding to a factorization of $x^2-bx+c$ into linear factors over $R$.
 If we have two such factorizations
 \[x^2-bx+c = (x-r)(x-s) = (x-t)(x-u),\]
 then we have $rs = c = tu = t(b-t) = t(r+s-t)$, so $(r-t)(s-t) = 0$.
 So since $R$ is a domain, we must have $r=t$ (and $s=u$) or $s=t$ (and $r=u$); either way, the two factorizations correspond to isomorphic closure data.
 
 Conversely, if $rs=0$ in $R$ with $r$ and $s$ nonzero, then we have
 \[(x-r)(x-s) = x\bigl(x-(r+s)\bigr),\]
 factorizations which correspond to two non-isomorphic $1$-closure data for $R[x]/(x^2-(r+s)x)$ over $R$.
\end{remark}

\begin{remark}
More generally, isomorphic $\permgroup{n_1}\times\dots\times\permgroup{n_k}$-closure data correspond to factorizations that differ only in the order of factors of the same degree.
The group of such reorderings is exactly the quotient by $\permgroup{n_1}\times\dots\times\permgroup{n_k}$ of its normalizer in $\permgroup{n}$---see \cref{group-actions}.
\end{remark}

\subsection{Parameterizing $G$-closure data}

%%%%%% Discriminants for monogenic algebras

Note that in the case of monogenic algebras, \cref{sqrt-disc} gives us the following criterion for a monogenic rank-$n$ algebra to have an $\altgroup{n}$-closure datum:

\begin{theorem}
 Let $R$ be a ring in which $2$ is a primoid non-zerodivisor, and let $A=R[x]/(f(x))$ be a monogenic rank-$n$ $R$-algebra.  
 Then $\altgroup{n}$-closure data for $A$ over $R$ correspond to square roots of the discriminant of $f$.
\end{theorem}

We wish to produce similar parameterizations of monogenic algebras' closure data for other groups than $\altgroup{n}$, where closure data for $A$ over $R$ correspond to solutions in $R$ of certain polynomial equations whose coefficients depend on $A$.
The main goal of this section is \cref{specialparametrize}, which abstractly allows one to produce sucha  parameterization whenever the pair (ring $R$, group $G$) forms a ``benign pair,'' to be defined below.

%%%%%% Definition of "benign" pair $(R,G)$.

Recall that an $R$-module $M$ is called \emph{faithful} if no nonzero element of $R$ acts as zero on $M$, and that an $R$-algebra $B$ is faithful as an $R$-module if and only if the structure map $R\to B$ is injective.

\begin{definition}
 Let $R$ be a ring and $G$ a subgroup of $\permgroup{n}$ for a fixed natural number $n$.
 We say that the pair $(R,G)$ is \emph{benign} if for every $R$-algebra $B$ with an action of $G$ by $R$-algebra homomorphisms, and for every $R$-algebra homomorphism $B^G\to R$, the resulting tensor product $B\otimes_{B^G} R$ is a faithful $R$-algebra.
\end{definition}

\begin{lemma}
 Either of the following two conditions is sufficient for the pair $(R,G)$ to be a benign pair:
 \begin{enumerate}
  \item $R$ is reduced.
  \item $\card{G}$ is a non-zerodivisor in $R$.
 \end{enumerate}
\end{lemma}

\begin{proof}
 Recall that given a homomorphism of rings $f:A\to B$, the corresponding map of schemes $\spec(B)\to\spec(A)$ is surjective if and only if the kernel of $f$ consists of nilpotents.
 Then if $R$ is reduced, injectivity of $R\to B\otimes_{B^G} R$ is equivalent to surjectivity of $\spec(B\otimes_{B^G} R)\to\spec(R)$.
 But this is guaranteed by surjectivity of $\spec(B)\to\spec(B^G)$, because $B^G\to B$ is injective and surjectivity of morphisms of schemes is preserved under base change.
 
 Now suppose instead that $\card{G}$ is a non-zerodivisor in $R$.
 Then consider the $B^G$-module homomorphism $B\to B^G$ sending $b\mapsto \sum_{g\in G}g.b$.
 On elements $b$ that are already fixed by $G$, each term in the sum is just $b$ again, so the composite
 \[B^G\to B \to B^G\]
 is  multiplication by $\card{G}$.
 After base changing along the given $R$-algebra homomorphism $B^G\to R$, then, we find that the composite
 \[R \to B\otimes_{B^G} R \to R\]
 is multiplication by $\card{G}$, which is injective since $\card{G}$ is a non-zerodivisor.
 Therefore the map $R\to B\otimes_{B^G} R$ must be injective as well.
\end{proof}

If $(R,G)$ is a benign pair with $G\subset\permgroup{n}$, then for every rank-$n$ $R$-algebra $A$ with a $G$-closure datum $\phi$, we find that the homomorphism $R\to A^{\otimes n}/\phi$ is injective.
 Not every pair is benign, however, and not every closure algebra is faithful:

\begin{example}
 Let $R=\Z/(9)$ and $G=\altgroup{3}$.
 Then the $R$-algebra $A=R[x]/(x^3)$ has a $G$-closure $\phi$ for which the map $R\to A^{\otimes 3}/\phi$ sends $3$ to $0$.
 
 Namely, since $2$ is a unit in $R$ we have a correspondence between $\altgroup{3}$-closure data for $A$ over $R$ and square roots in $R$ of the discriminant of $A$, which vanishes.
 If we choose the square root $3$ of $0$ in $R$, the corresponding $\altgroup{3}$-closure datum sends $\gamma(1,x,x^2) - \gamma(1,x^2,x)$ to $3\in R$, and since the sum $\gamma(1,x,x^2) + \gamma(1,x^2,x)$ must be sent to zero by \cref{monogenic-A3-example}, we find that $\gamma(1,x,x^2)\mapsto 6$ and $\gamma(1,x^2,x)\mapsto 3$.
 
 Then in the closure algebra $A^{\otimes 3}/\phi$, we find that
 \begin{align*}
  \gamma(1,x^2,x) &= 1\otimes x^2 \otimes x + x^2 \otimes x \otimes 1 + x \otimes 1 \otimes x^2\\
  &= (1\otimes x^2 \otimes 1) (-x\otimes 1 \otimes 1-1\otimes x\otimes 1)\\
  &\quad + x^2\otimes x \otimes 1 \\
  &\quad + (x\otimes 1 \otimes 1)(-x\otimes 1 \otimes 1-1\otimes x\otimes 1)^2,
  \intertext{using the relation $x\otimes 1\otimes 1 + 1\otimes x\otimes 1 + 1\otimes 1\otimes x = \trace_{A}(x) = 0$,}
  &= -x\otimes x^2\otimes 1 + x^2\otimes x\otimes 1 + x\otimes x^2\otimes 1 + 2x^2\otimes x\otimes 1\\
  &= 3x^2 \otimes x\otimes 1.
 \end{align*}
Therefore $3 = 3x^2\otimes x\otimes 1$.
Multiplying both sides by $x^2\otimes x\otimes 1$, we find that $3x^2\otimes x \otimes 1 = 0$.
Thus by transitivity, $3=0$ in the closure algebra.
\end{example}

\begin{remark}
 Even though the pair $(R,\permgroup{n})$ is not always benign, the $\permgroup{n}$-closure of a rank-$n$ $R$-algebra $A$ is always faithful.
 If we use the Ferrand homomorphism to equip the $R$-module $\extpower^n A$ with an $\fixpower{A}{n}{\permgroup{n}}$-module structure, then the defining surjection $A^{\otimes n}\to\extpower^n A$ is actually a $\fixpower{A}{n}{\permgroup{n}}$-module homomorphism by \cite[Lemma 4.2]{Bie15}.
 Then tensoring with $R$ over $\fixpower{A}{n}{\permgroup{n}}$ gives a surjection
 \[\gclose{A}{n}{\Ferrand{A}} \ =\  A^{\otimes n} \ \midotimes_{\mathclap{\fixpower{A}{n}{\permgroup{n}}}}\ R \ \onto\  \extpower^n A\ \midotimes_{\mathclap{\fixpower{A}{n}{\permgroup{n}}}}\ R \ \cong\  \extpower^n A.\]
 Since $\extpower^n A$ is a locally free $R$-module of rank $1$, it is faithful, and therefore $\gclose{A}{n}{\Ferrand{A}}$ must be too.
\end{remark}

%%%%%% Simpler description for benign pairs:
But supposing that the pair $(R,G)$ \emph{is} benign, then we obtain the following parameterization of $G$-closure data for monogenic $R$-algebras:

\begin{lemma}\label{specialparametrize}
 Let $R$ be a ring and $A$ be a monogenic $R$-algebra with generator $a$.
 Given a closure datum $(G,\phi)$ for $A$ over $R$, we may compose $\phi$ with the projection $\fixpower{R[x]}{n}{G}\to\fixpower{A}{n}{G}$ to obtain an $R$-algebra homomorphism $\fixpower{R[x]}{n}{G}\to R$ such that $e_k(x)\mapsto s_k(a)$ for all $k\in\set{n}$.
 
 If $G$ is a subgroup of $\permgroup{n}$ for which $(R,G)$ is benign, then this operation forms a one-to-one correspondence between $G$-closure data for $A$ over $R$ and such homomorphisms $\fixpower{R[x]}{n}{G}\to R$.
\end{lemma}

In particular, if we can present $\fixpower{R[x]}{n}{G}$ as an algebra over $\fixpower{R[x]}{n}{\permgroup{n}}$, then $G$-closure data for $A$ over $R$ will correspond to solutions in $R$ of a list of polynomial equations, the way $\altgroup{n}$-closure data correspond to square roots of the discriminant.
In the next section, we will do just that in the case $G=\dihedralgroup{4} = \langle(13),(1234)\rangle\subset\permgroup{4}$.
Since the publication of this argument in the author's PhD thesis, Riccardo Ferrario has produced similar results in \cite{Fer14} for the cases $\kleinfour = \langle(12)(34),(13)(24)\rangle$ and $\cyclicgroup{4} = \langle(1234)\rangle$.
The parameterization of $\kleinfour$-closure data is very similar to the one that follows for $\dihedralgroup{4}$---they correspond to \emph{splittings} of the cubic resolvent instead of roots---but so far there is no nice interpretation for the parameterization of $\cyclicgroup{4}$-closure data.

\begin{proof}[Proof of \cref{specialparametrize}]
 Suppose that $(G,\phi)$ is a closure datum for $A$ over $R$.
 Then under the composite $\fixpower{R[x]}{n}{G}\to\fixpower{A}{n}{G}\to R$, we have $e_k(x)\mapsto e_k(a)\mapsto s_k(a)$.
 
 Now conversely, suppose that $\fixpower{R[x]}{n}{G}\to R$ is an $R$-algebra homomorphism sending $e_k(x)$ to $s_k(a)$ for all $k\in\set{n}$, and use the hypothesis that $(R,G)$ is a benign pair to obtain that the resulting tensor product
 \[T:= R[x]^{\otimes n}\ \midotimes_{\mathclap{\fixpower{R[x]}{n}{G}}}\ R\]
 is a faithful $R$-algebra.
 We will fill in the two dashed arrows in the following commutative diagram:
 \[\begin{tikzcd}
 \fixpower{R[x]}{n}{G} \arrow[two heads]{r} \arrow[bend left, two heads]{rr}\arrow[hook]{d} & \fixpower{A}{n}{G}\arrow[dashed]{r}\arrow[hook]{d} & R\arrow[hook]{d}\\
 R[x]^{\otimes n}\arrow[bend right, two heads]{rr}\arrow[two heads]{r} & A^{\otimes n}\arrow[dashed]{r} & T\\
 \end{tikzcd}\]
  For the existence of the lower dashed arrow, notice that for each $i\in\set{n}$, the image of $\charpoly{a}(\conjugate{x}{i})$ under the map $R[x]^{\otimes n}\to T$ is
  \begin{align*}
  \charpoly{a}(\conjugate{x}{i}) &= \sum_{k=0}^n(-1)^ks_k(a)(\conjugate{x}{i})^{n-k} \\
  &= \sum_{k=0}^n (-1)^ke_k(x)(\conjugate{x}{i})^{n-k} = \prod_{j=1}^n(\conjugate{x}{i}-\conjugate{x}{j}) = 0,
  \end{align*}
  so each component of the map factors through the projection $R[x]\onto A\cong R[x]/(\charpoly{a}(x))$.
  
  Then the existence of the upper dashed arrow follows elementarily: we have the composite $\fixpower{A}{n}{G}\into A^{\otimes n}\to T$, and by the commutativity of the rest of the diagram its image is contained in the subring $R$.
  Thus we obtain the existence of a (necessarily unique) map $\fixpower{A}{n}{G}\to R$ commuting with the maps from $\fixpower{R[x]}{n}{G}$.
  In particular, this map is a $G$-closure datum for $A$ over $R$, because $e_k(x)$ in $\fixpower{R[x]}{n}{G}$ is sent to $e_k(a)$ in $\fixpower{A}{n}{G}$ and $s_k(a)$ in $R$.
\end{proof}

%%%%%% D4 closure data

\subsection{$\dihedralgroup{4}$-closure data}

A classical result of Galois theory is that the Galois group of a separable irreducible quartic polynomial
\[f(x) = x^4 - s_1x^3 + s_2 x^2 - s_3 x + s_4\] 
is contained in the permutation group $\dihedralgroup{4} = \langle(13),(1234)\rangle\subset\permgroup{4}$ if and only if that polynomial's \emph{cubic resolvent}
\[m(y) = y^3 - (s_2)y^2 + (s_1s_3-4s_4)y - (s_1^2s_4 - 4s_2s_4+s_3^2)\] 
has a root in the base field.
In this section, we prove the following generalization:

\begin{theorem}\label{main-D4}
 Let $R$ be a ring  and let $A=R[x]/(f(x))$ be a monogenic rank-$4$ $R$-algebra.  
 Then $\dihedralgroup{4}$-closure data for $A$ over $R$ correspond to roots of $f$'s cubic resolvent in $R$.
\end{theorem}

We will do so by first giving generators and relations for $\fixpower{R[x]}{4}{\dihedralgroup{4}}$ as an algebra over $\fixpower{R[x]}{4}{\permgroup{4}}$, and then using this presentation to show that if $R$ is reduced, then $\dihedralgroup{4}$-closure data of $R[x]/(f(x))$ correspond to roots in $R$ of the cubic resolvent of $f(x)$.
Finally, we will carefully lift the condition that $R$ be reduced.

%%%%%% Generators for R[x]^{\dihedralgroup{4}} as an R[x]^{\permgroup{4}} algebra

\begin{lemma}
 The ring $\fixpower{\Z[x]}{4}{\dihedralgroup{4}}$ is a free $\fixpower{\Z[x]}{4}{\permgroup{4}}$-module with basis $\{1, \Lambda, \Lambda^2\}$, where $\Lambda = \conjugate{x}{1}\conjugate{x}{3} + \conjugate{x}{2}\conjugate{x}{4}$.
\end{lemma}

\begin{proof}
 First, we fix some helpful notation.
 We will write $x_1,x_2,x_3,x_4$ for the four conjugates $\conjugate{x}{1}, \conjugate{x}{2}, \conjugate{x}{3}, \conjugate{x}{4}$ in $R[x]^{\otimes 4}$, identifying the latter $R$-algebra with $R[x_1,x_2,x_3,x_4]$.
 If $p\in \Z[x_1,x_2,x_3,x_4]^{\dihedralgroup{4}}$, then we denote the polynomial $(14).p = (23).p$ in $\Z[x_1,x_2,x_3,x_4]$ by $p'$, and the polynomial $(12).p = (34).p$ by $p''$.  
 Each transposition permutes $\{p,p',p''\}$:
 \begin{align*}
  (23)\text{ and }(14)\text{ interchange }p&\leftrightarrow p'\text{ and fix }p''\\
  (12)\text{ and }(34)\text{ interchange }p&\leftrightarrow p''\text{ and fix }p'\\
  (13)\text{ and }(24)\text{ interchange }p'&\leftrightarrow p''\text{ and fix } p.
 \end{align*}
 If $p\in \Z[x_1,x_2,x_3,x_4]^{\dihedralgroup{4}}$ and any two of $\{p,p',p''\}$ are equal, then $p\in\Z[x_1,x_2,x_3,x_4]^{\permgroup{4}}$.  
 In particular, $\Lambda, \Lambda',$ and $\Lambda''$ are all distinct:
 \begin{align*}
  \Lambda - \Lambda' &= (x_1-x_4)(x_3-x_2)\\
  \Lambda - \Lambda'' &= (x_1-x_2)(x_3-x_4)\\
  \Lambda' - \Lambda'' &= (x_1-x_3)(x_2-x_4).
 \end{align*}
 
 First, we show that $1$, $\Lambda$, and $\Lambda^2$ are $\Z[x_1,x_2,x_3,x_4]^{\permgroup{4}}$-linearly independent.  
 Suppose $q\Lambda^2 + r\Lambda + s = 0$, with $q,r,s\in\Z[x_1,x_2,x_3,x_4]^{\permgroup{4}}$.  
 Then $0 = q\Lambda'^2 + r\Lambda' + s = q\Lambda''^2 + r\Lambda'' + s$, so
 \[
 0 = \frac{q(\Lambda'^2 - \Lambda''^2) + r(\Lambda' - \Lambda'')}{\Lambda' - \Lambda''} = q(\Lambda' + \Lambda'') + r = -q\Lambda + (r + q(\Lambda + \Lambda' + \Lambda'')).
 \]
 Therefore $0 = -q\Lambda' + (r + q(\Lambda + \Lambda' + \Lambda'')) = -q\Lambda'' + (r + q(\Lambda + \Lambda' + \Lambda''))$, so $q(\Lambda'-\Lambda'')=0$, and $q=0$.  
 Then $0 = q(\Lambda' + \Lambda'') + r$ implies that $r=0$, and $0 = q\Lambda^2 + r\Lambda + s$ implies that $s=0$.
 
 To show that $1$, $\Lambda$, and $\Lambda^2$ also generate $\Z[x_1,x_2,x_3,x_4]^{\dihedralgroup{4}}$ as a module over $\Z[x_1,x_2,x_3,x_4]^{\permgroup{4}}$, we use the following observation:
 If $p\in \Z[x_1,x_2,x_3,x_4]^{\dihedralgroup{4}}$, then $p' \equiv p''$ modulo either $(x_1-x_3)$ or $(x_2-x_4)$, so $p'-p''$ must contain factors of both $(x_1-x_3)$ and $(x_2-x_4)$.  
 Since $\Z[x_1,x_2,x_3,x_4]$ is a unique factorization domain, we find that $p'-p''$ is a multiple of $\Lambda'-\Lambda''$.
 In fact, the ratio $\rho = \frac{p'-p''}{\Lambda'-\Lambda''}$ also belongs to $\Z[x_1,x_2,x_3,x_4]^{\dihedralgroup{4}}$, since it is fixed by $(13)$ and $(1234)=(12)(23)(34)$:
 \[
 \begin{tikzcd}[row sep=tiny]
 \ds\rho = \frac{p'-p''}{\Lambda'-\Lambda''} \arrow[mapsto]{r}{(13)} & \ds\frac{p''-p'}{\Lambda''-\Lambda'} = \rho,\\
 \rho = \ds\frac{p'-p''}{\Lambda'-\Lambda''} \arrow[mapsto]{r}{(34)} & \ds\frac{p'-p}{\Lambda'-\Lambda} \arrow[mapsto]{r}{(23)} & \ds\frac{p-p'}{\Lambda-\Lambda'} \arrow[mapsto]{r}{(12)} & \ds\frac{p''-p'}{\Lambda''-\Lambda'} = \rho.
 \end{tikzcd}\]
 Thus we can apply the same procedure to $\rho$ as we did to $p$; set
 \[
 q = -\frac{\rho'-\rho''}{\Lambda'-\Lambda''}.
 \]
 We claim that $q\in\Z[x_1,x_2,x_3,x_4]^{\permgroup{4}}$.  
 In fact, we can write
 \begin{align*}
  -q &= \frac{\rho'-\rho''}{\Lambda'-\Lambda''}\\
  &= \frac{\ds\frac{p-p''}{\Lambda-\Lambda''}-\frac{p'-p}{\Lambda'-\Lambda}}{\Lambda'-\Lambda''}\\
  &= \frac{(p-p'')(\Lambda-\Lambda') - (p-p')(\Lambda - \Lambda'')}{(\Lambda-\Lambda')(\Lambda-\Lambda'')(\Lambda'-\Lambda'')}\\
  &= \frac{(p- p')\Lambda'' + (p'' - p)\Lambda'+ (p'-p'')\Lambda}{(\Lambda-\Lambda')(\Lambda-\Lambda'')(\Lambda'-\Lambda'')}.
 \end{align*}
 Every transposition changes the sign of both the numerator and the denominator of the right-hand side, so $q$ is fixed by the action of $\permgroup{4}$.  
 Now set
 \[
 r = \rho - q(\Lambda' + \Lambda'') \in \Z[x_1,x_2,x_3,x_4]^{\dihedralgroup{4}}.
 \] 
 Again, we claim that $r\in\Z[x_1,x_2,x_3,x_4]^{\permgroup{4}}$, and calculate
 \begin{align*}
  r &= \frac{p'-p''}{\Lambda'-\Lambda''} + \frac{((p- p')\Lambda'' + (p'' - p)\Lambda'+ (p'-p'')\Lambda)(\Lambda' + \Lambda'')}{(\Lambda-\Lambda')(\Lambda-\Lambda'')(\Lambda'-\Lambda'')}\\
  &= \frac{(p-p')\Lambda''^2 + (p''-p)\Lambda'^2 + (p'-p'')\Lambda^2}{(\Lambda-\Lambda')(\Lambda-\Lambda'')(\Lambda'-\Lambda'')}.
 \end{align*}
 Once again, every transposition changes the sign of both numerator and denominator, so $r$ is fixed by $\permgroup{4}$. 
  Finally, set
 \[
 s = p - q\Lambda^2 - r\Lambda \in \Z[x_1,x_2,x_3,x_4]^{\dihedralgroup{4}}.
 \]
 Again, we claim that $s\in \Z[x_1,x_2,x_3,x_4]^{\permgroup{4}}$:
 \begin{align*}
  s &= p - q\Lambda^2 - r\Lambda\\
  &= p + \frac{(p- p')\Lambda'' + (p'' - p)\Lambda'+ (p'-p'')\Lambda}{(\Lambda-\Lambda')(\Lambda-\Lambda'')(\Lambda'-\Lambda'')}\Lambda^2\\
  &\qquad - \frac{(p-p')\Lambda''^2 + (p''-p)\Lambda'^2 + (p'-p'')\Lambda^2}{(\Lambda-\Lambda')(\Lambda-\Lambda'')(\Lambda'-\Lambda'')}\Lambda\\
  &= \frac{p(\Lambda'-\Lambda'')\Lambda'\Lambda'' + p'(\Lambda''-\Lambda)\Lambda\Lambda'' + p''(\Lambda-\Lambda')\Lambda\Lambda'}{(\Lambda-\Lambda')(\Lambda-\Lambda'')(\Lambda'-\Lambda'')}.
 \end{align*}
 Once again, the numerator and denominator each change sign under the action of any transposition, so $s$ is fixed by $\Z[x_1,x_2,x_3,x_4]^{\permgroup{4}}$.  
 Thus we have written $p = q\Lambda^2 + r\Lambda + s$ with $q,r,s\in \Z[x_1,x_2,x_3,x_4]^{\permgroup{4}}$, as desired.
\end{proof}

\begin{corollary}\label{lambdaalgebragenerate}
 Let $R$ be a ring.  
 Then
 \[
  \fixpower{R[x]}{4}{\dihedralgroup{4}} \cong \fixpower{R[x]}{4}{\permgroup{4}}[y]/((y-\Lambda)(y-\Lambda')(y-\Lambda''))
  \]
 as $\fixpower{R[x]}{4}{\permgroup{4}}$-algebras.
\end{corollary}

\begin{proof}
 The isomorphism from right to left is given by $y\mapsto \Lambda$. 
 This homomorphism is bijective since it maps the $\fixpower{R[x]}{4}{\permgroup{4}}$-module basis $\{1,y,y^2\}$ to the module basis $\{1,\Lambda,\Lambda^2\}$.
\end{proof}

%%%%%% Description of $\dihedralgroup{4}$-closures

\begin{proof}[Proof of \cref{main-D4}]
 First we prove this in the case that $R$ is reduced, so that $(R,\dihedralgroup{4})$ is a benign pair.
 Then by \cref{specialparametrize}, isomorphism classes of $\dihedralgroup{4}$-closures of $A$ over $R$ correspond to $R$-algebra homomorphisms $\fixpower{R[x]}{4}{\dihedralgroup{4}}\to R$ sending $e_k(x)\mapsto s_k$.
 If we denote the $R$-algebra map $\fixpower{R[x]}{4}{\permgroup{4}}\to R$ sending $e_k(x)\mapsto s_k$ by $\phi$, then such homomorphisms in turn correspond to $R$-algebra homomorphisms to $R$ from $\fixpower{R[x]}{4}{\dihedralgroup{4}}\otimes_{\fixpower{R[x]}{4}{\permgroup{4}}} R$, which by \cref{lambdaalgebragenerate} is 
 \[\begin{gathered}
 R[y]/\bigl(y^3 - \phi(\Lambda+\Lambda'+\Lambda'')y^2 + \phi(\Lambda\Lambda' + \Lambda\Lambda'' + \Lambda'\Lambda'')y - \phi(\Lambda\Lambda'\Lambda'')\bigr)\\
 = R[y]/\bigl(y^3 - (s_2)y^2 + (s_1s_3-4s_4)y - (s_1^2s_4 - 4s_2s_4+s_3^2)\bigr),
 \end{gathered}\]
the cubic resolvent algebra.
 
 One way of phrasing the conclusion of \cref{specialparametrize} is that every $R$-algebra homomorphism
 \[\fixpower{R[x]}{4}{\dihedralgroup{4}}\ \midotimes_{\mathclap{\fixpower{R[x]}{4}{\permgroup{4}}}}\ R \to R\]
  factors through the quotient map
 \[\fixpower{R[x]}{4}{\dihedralgroup{4}}\ \midotimes_{\mathclap{\fixpower{R[x]}{4}{\permgroup{4}}}}\ R \onto \fixpower{A}{4}{\dihedralgroup{4}}\ \midotimes_{\mathclap{\fixpower{A}{4}{\permgroup{4}}}}\ R\]
 to give a $\dihedralgroup{4}$-closure datum $\fixpower{A}{4}{\dihedralgroup{4}}\otimes_{\fixpower{A}{4}{\permgroup{4}}} R\to R$. 
 We now show that this holds even if $R$ is not reduced.
 
 Indeed, consider the universal monogenic rank-$4$ algebra $R_0\to A_0$ given by
 \begin{align*}
  R_0 &= \Z[S_1,S_2,S_3,S_4],\qquad\\
   A_0 &= R_0[x]/(x^4-S_1x^3+S_2x^2-S_3x+S_4),
 \end{align*}
 with the $S_i$ formal indeterminates.
 Then the algebra $R\to A$ is the base change of $R_0\to A_0$ along the ring homomorphism $R_0\to R$ sending each $S_i$ to $s_i$.
 
 What we have already shown is that if $R$ is any \emph{reduced} $R_0$-algebra, then every $R_0$-algebra homomorphism
 \[\fixpower{R_0[x]}{4}{\dihedralgroup{4}}\ \midotimes_{\mathclap{\fixpower{R_0[x]}{4}{\permgroup{4}}}}\ R_0 \to R\]
 factors through the quotient map
 \[\fixpower{R_0[x]}{4}{\dihedralgroup{4}}\ \midotimes_{\mathclap{\fixpower{R_0[x]}{4}{\permgroup{4}}}}\ R_0 \onto \fixpower{A_0}{4}{\dihedralgroup{4}}\ \midotimes_{\mathclap{\fixpower{A_0}{4}{\permgroup{4}}}}\ R_0.\]
 
 In particular, we can take $R$ itself to be the tensor product
 \[\begin{gathered}R:=\ \fixpower{R_0[x]}{4}{\dihedralgroup{4}}\ \midotimes_{\mathclap{\fixpower{R_0[x]}{4}{\permgroup{4}}}}\ R_0\\
 \cong R_0[y]/\bigl(y^3 - (S_2)y^2 + (S_1S_3-4S_4)y - (S_1^2S_4 - 4S_2S_4+S_3^2)\bigr),
 \end{gathered}\]
 which is reduced because $y^3 - (S_2)y^2 + (S_1S_3-4S_4)y - (S_1^2S_4 - 4S_2S_4+S_3^2)$ is an irreducible element of the polynomial ring $\Z[S_1,S_2,S_3,S_4,y]$---if it were not, every cubic resolvent would have a root, and every separable quartic polynomial would have Galois group contained in $\dihedralgroup{4}$.
 
 Therefore the identity homomorphism on $R$ factors through its quotient $\fixpower{A_0}{4}{\dihedralgroup{4}}\otimes_{\fixpower{A_0}{4}{\permgroup{4}}} R_0$, so the quotient map must in fact be an isomorphism:
 \[\fixpower{R_0[x]}{4}{\dihedralgroup{4}}\ \midotimes_{\mathclap{\fixpower{R_0[x]}{4}{\permgroup{4}}}}\  R_0 \ \cong\  \fixpower{A_0}{4}{\dihedralgroup{4}}\ \midotimes_{\mathclap{\fixpower{A_0}{4}{\permgroup{4}}}}\ R_0.\]
 Changing base to a general ring $R$, then, we find that
 \[\fixpower{R[x]}{4}{\dihedralgroup{4}}\ \midotimes_{\mathclap{\fixpower{R[x]}{4}{\permgroup{4}}}}\ R\ \cong\  \fixpower{A}{4}{\dihedralgroup{4}}\ \midotimes_{\mathclap{\fixpower{A}{4}{\permgroup{4}}}}\ R,\]
 so the conclusion of \cref{specialparametrize} holds even if $R$ is not reduced, and $\dihedralgroup{4}$-closure data for $R[x]/(f(x))$ correspond to roots of $f$'s cubic resolvent.
\end{proof}

%%%%%%

\bibliographystyle{plain}
\bibliography{RefList}

\end{document}